\documentclass[11pt]{article}
\usepackage[letterpaper,margin=1in]{geometry}
\usepackage[T1]{fontenc}
\usepackage{lmodern}
\usepackage{amsmath,amssymb,amsthm,mathtools}
\usepackage{microtype}
\usepackage{algorithm,algpseudocode}
\usepackage{aliascnt}
\usepackage{graphicx}
\usepackage[authoryear,round]{natbib}
\usepackage[colorlinks=true,linkcolor=blue,citecolor=blue,urlcolor=blue]{hyperref}
\usepackage[nameinlink,noabbrev]{cleveref}
\floatstyle{ruled}
\newfloat{procedure}{htbp}{lop}[section]
\floatname{procedure}{Procedure}
\crefname{procedure}{procedure}{procedures}
\Crefname{procedure}{Procedure}{Procedures}
\makeatletter
\providecommand{\theHALG@line}{}
\renewcommand{\theHALG@line}{\@captype.\csname theH\@captype\endcsname.\arabic{ALG@line}}
\makeatother
\theoremstyle{plain}
\newtheorem{theorem}{Theorem}[section]
\newaliascnt{lemma}{theorem}
\newtheorem{lemma}[lemma]{Lemma}
\aliascntresetthe{lemma}
\crefname{theorem}{theorem}{theorems}
\Crefname{theorem}{Theorem}{Theorems}
\crefname{lemma}{lemma}{lemmas}
\Crefname{lemma}{Lemma}{Lemmas}
\theoremstyle{definition}
\newaliascnt{assumption}{theorem}
\newtheorem{assumption}[assumption]{Assumption}
\aliascntresetthe{assumption}
\crefname{assumption}{assumption}{assumptions}
\Crefname{assumption}{Assumption}{Assumptions}
\theoremstyle{remark}
\newaliascnt{remark}{theorem}
\newtheorem{remark}[remark]{Remark}
\aliascntresetthe{remark}
\crefname{remark}{remark}{remarks}
\Crefname{remark}{Remark}{Remarks}

\theoremstyle{plain}
\AtBeginDocument{%
  \setlength{\abovedisplayskip}{8pt plus 2pt minus 2pt}%
  \setlength{\belowdisplayskip}{8pt plus 2pt minus 2pt}%
  \setlength{\abovedisplayshortskip}{3pt plus 1pt}%
  \setlength{\belowdisplayshortskip}{5pt plus 2pt minus 2pt}%
}
\allowdisplaybreaks[1]
\hypersetup{pdftitle={Regularized Newton Methods via Proximal Point Frameworks},pdfauthor={Chang He and Shuzhong Zhang},pdfsubject={Proximal Newton subproblems, curvature certificates, and universal scale adaptation}}

\newcommand{\bR}{\mathbb{R}}
\newcommand{\epsg}{\varepsilon_g}
\newcommand{\epsH}{\varepsilon_H}
\newcommand{\op}{\mathrm{op}}
\newcommand{\inner}[2]{\langle #1,#2\rangle}
\newcommand{\anorm}[1]{\|#1\|_{k,\mu}}
\DeclareMathOperator*{\argmin}{argmin}

\newcommand{\cO}{\mathcal{O}}
\Crefname{appendix}{Appendix}{Appendix}
\title{Proximal Point Newton Methods for Nonconvex Optimization}
\author{
{Chang He} \thanks{Department of Industrial and Systems Engineering, University of Minnesota. \texttt{hec@umn.edu}}
\and
{Shuzhong Zhang} \thanks{Department of Industrial and Systems Engineering, University of Minnesota. \texttt{zhangs@umn.edu}}
}
\date{}
\begin{document}
\maketitle

\begin{abstract}
Tensor methods offer improved complexity guarantees for nonconvex optimization, but evaluating high-order derivatives and solving the corresponding subproblems can be computationally expensive. In this paper, we develop a proximal point Newton framework that exploits third-derivative smoothness using only gradients and Hessians. Under Hessian Lipschitz continuity, a single-loop proximal point Newton method that alternates between negative-curvature directions and Newton steps finds an $(\epsg,\epsH)$-approximate second-order stationary point within $\cO(\max\{\epsg^{-3/2},\epsH^{-3}\})$ second-order oracle evaluations. When the third derivative is also Lipschitz continuous, we introduce an accelerated proximal point Newton method and obtain a dimension-free complexity of $\widetilde\cO(\max\{\epsg^{-4/3},\epsH^{-2}\})$, matching that of the third-order tensor methods up to logarithmic factors without constructing tensors or solving tensor subproblems. We further develop an adaptive inexact Newton--CG variant that requires no prior knowledge of Lipschitz constants and preserves the improved oracle complexity with high probability. Under uniformly bounded queried Hessians, this variant achieves a Hessian-vector product complexity of $\widetilde\cO(\max\{\epsg^{-5/3},\epsH^{-5/2}\})$, improving upon the previous $\widetilde\cO(\max\{\epsg^{-3}\epsH^{5/2},\epsH^{-7/2}\})$ result attained by existing second-order methods such as cubic regularized Newton method and Newton--CG method. Preliminary numerical experiments are encouraging.
\end{abstract}

\medskip

{\bf Keywords:} nonconvex optimization; proximal point method; Newton--CG method; oracle complexity.

\medskip

{\bf MSC Classification (2020):} 90C26, 90C30, 49M15, 90C60.

\section{Introduction}\label{sec:introduction}
Tensor methods have attracted wide attention in recent decades because high-order derivative information can improve the worst-case evaluation complexity of smooth optimization. Specifically, for nonconvex optimization, when the $p$th derivative is Lipschitz continuous, the adaptive regularization method based on $p$th-order Taylor models can find an $(\epsg,\epsH)$-approximate second-order stationary point, namely a point $x$ satisfying $\|\nabla f(x)\|\le\epsg$ and $\lambda_{\min}(\nabla^2 f(x)) > -\epsH$ with $\lambda_{\min}(\cdot)$ denoting the smallest eigenvalue, within at most $\cO(\max\{\epsg^{-(p+1)/p},\epsH^{-(p+1)/(p-1)}\})$ times call for $p$th-order oracle evaluations\footnote{A $p$th-order oracle evaluation returns the function value and derivative information up to order $p$, i.e., $(f(x), \nabla f(x), \ldots, \nabla^p f(x))$.} \citep{birgin2017worst,cartisgouldtoint2020second}. In convex optimization, an accelerated $p$th-order tensor method finds an $\varepsilon$-approximate optimal solution within $\cO(\varepsilon^{-1/(p+1)})$ $p$th-order oracle evaluations \citep{nesterov2021implementable}. However, these benefits come at the cost of needing to solve increasingly difficult subproblems. Classical second-order methods, such as trust-region methods \citep{conn2000trust} and cubic regularized Newton methods \citep{nesterovpolyak2006}, admit subproblems for which globally optimal solutions can be computed efficiently using matrix computations. In contrast, for a regularized $p$th-order Taylor model with $p\ge3$, finding a global solution  for the tensor subproblem itself is computationally intractable in general. Therefore, efficient implementation is a critical issue for tensor methods.

In convex optimization, \citet{nesterov2021implementable} showed that appropriately regularized third-order Taylor models are convex and developed an efficient non-Euclidean method for solving the resulting subproblems. Subsequently, \citet{nesterov2021superfast,kamzolov2020near} replaced third-derivative actions by directional finite differences of gradients, thereby obtaining a second-order implementation of the third-order method. Related inexact and high-order proximal point frameworks provide alternative lower-order implementations of tensor methods \citep{doikovnesterov2020,grapiglia2021inexact,nesterov2022inexact,nesterov2023proximal,grapiglianesterov2023}. From a different perspective, \citet{ahmadi2024higher} constructed SOS-convex approximations of $p$th-order Taylor models that can be minimized through semidefinite programming. More recently, \citet{doikovmishchenkonesterov2024} exploited the Lipschitz continuity of the third derivative directly through gradient-regularized Newton steps, without explicitly constructing the regularized third-order Taylor model. Across these developments, convexity is the key structural property that enables the implementability of tensor methods.

In nonconvex optimization, the lack of convexity makes the regularized third-order subproblem substantially more difficult to solve. A natural approach is to replace the original regularized third-order Taylor model with auxiliary models having more tractable structures. Along this direction, \citet{cartiszhu2026,cartiszhu2026second} developed quadratic--quartic and cubic--quartic regularization schemes that iteratively minimize the regularized third-order subproblem. Recent studies have also established exact semidefinite programming reformulations for certain classes of these models \citep{zhu2026sufficiently,zhou2026tight}. \citet{zhu2025global} characterized global minimizers of quartically regularized third-order subproblems and developed specialized algorithms that exploit diagonal and low-rank tensor structures. More recently, \citet{cartishauserliuwelzelzhu2026} studied practical adaptive third-order implementations that further improve the selection of regularization parameters and the solution of subproblems. Despite the progress in subproblem tractability, the resulting methods still require explicit access to third-order derivative information. Moreover, solving a sequence of structured auxiliary subproblems may introduce additional computational costs beyond the outer evaluation complexity \citep{cartiszhu2026}.

The above considerations motivate a complementary line of research that approximates high-order derivatives using lower-order information. \citet{welzelhauser2024} developed high-order secant updates that generalize quasi-Newton constructions to derivative tensors. \citet{cartis2025global} then incorporated secant updates and lazy tensor approximations into adaptive regularization frameworks, with periodic exact or finite-difference refreshes. In particular, they show that an $\cO(d\max\{\epsg^{-4/3},\epsH^{-2}\})$ second-order oracle evaluations are required to find an $(\epsg,\epsH)$-approximate second-order stationary point for their third-order variant, where $d$ is the problem dimension. To reduce the cost of constructing these approximations, \citet{doikovgrapiglia2026} reused finite-difference tensors across multiple iterations. As a result, their general order-$(p-1)$ implementation requires at most $\cO(d^{1/p}\epsg^{-(p+1)/p}+d)$ $(p-1)$th-order oracle evaluations to find an $\epsg$-approximate first-order stationary point. Recent advances lower the derivative order required of the objective function. Nevertheless, existing oracle complexity results have an explicit dependence on the dimension, and tensor approximation does not eliminate the need to solve the resulting regularized $p$th-order subproblems.

In this paper, we bypass tensor subproblems altogether. Previous work has mainly focused on regularized third-order Taylor models, which inevitably require tensor computations and the solution of the associated subproblems. Instead, we develop a \textit{pure second-order method} based on the proximal point framework. At each iteration $k$, we apply either a negative-curvature direction or a Newton direction to solve the proximal point subproblem $\min_{x \in \bR^d} \Phi_k(x) \coloneqq f(x) + \mu \|x - x^k\|^2$ with a suitable regularization parameter $\mu$. Therefore, our method queries only gradients and Hessians, and its matrix computations consist of eigen-decompositions and linear system solves. Under Lipschitz continuity of the Hessian and the third derivative, our method achieves a dimension-free $\widetilde \cO(\max\{\epsg^{-4/3},\epsH^{-2}\})$ second-order oracle complexity for finding an $(\epsg,\epsH)$-approximate second-order stationary point, matching the complexity guarantee of the third-order tensor methods up to logarithmic factors. To the best of our knowledge, this is the first result to achieve such a guarantee using a second-order method. We then develop an adaptive version based on Newton-CG with inexact eigen-decompositions and linear system solves, which establishes a Hessian-vector product complexity of $\widetilde\cO(\max\{\epsg^{-5/3},\epsH^{-5/2}\})$. Finally, numerical experiments demonstrate the computational efficiency of our second-order method.

\section{Preliminaries}\label{sec:preliminaries}
Let $\mathbb{R}^d$ denote the $d$-dimensional Euclidean space. In this paper, we study the unconstrained minimization problem
\begin{equation}\label{eq:problem}
\min_{x\in\bR^d} f(x), 
\end{equation}
where $f$ is potentially nonconvex and $f_{\inf} \coloneqq \inf_{x\in\bR^d}f(x)>-\infty$. We aim to find an $(\epsg,\epsH)$-approximate second-order stationary point that satisfies $\|\nabla f(x)\|\le\epsg$ and $\lambda_{\min}(\nabla^2 f(x)) > -\epsH$, where $\lambda_{\min}(\cdot)$ denotes the smallest eigenvalue. The identity matrix is denoted by $I$. A second-order oracle evaluation returns $(f(x),\nabla f(x),\nabla^2 f(x))$. We use the big-O notation $\mathcal{O}(\cdot)$ to hide absolute constants and the notation $\tilde{\mathcal{O}}(\cdot)$ to hide polylogarithmic factors. The standard inner product is denoted by $\langle \cdot, \cdot \rangle$. We use $\|\cdot\|$ to denote the Euclidean norm for vectors and $\|\cdot\|_{\op}$ to denote the induced operator norm for matrices, respectively. We use $\nabla^3 f(x)$ to denote the third derivative of $f$ at point $x$, which is a symmetric third-order tensor, and its operator norm is defined as $\|\nabla^3 f(x)\|_{\op}\coloneqq \sup_{\|u\|,\|v\|,\|w\|\le1}|\nabla^3 f(x)[u,v,w]|$. The notation $\nabla^3 f(x)[u,\bullet,\bullet]$ denotes the corresponding symmetric matrix. For simplicity, we use the notation $\nabla^3 f(x)[u]^{\otimes 3}$ whenever all three vectors are the same. We make the following assumptions through the paper.

\begin{assumption}\label{ass:hessian}
The Hessian of the objective function $f$ is $L_H$-Lipschitz continuous, i.e., $\|\nabla^2 f(x)-\nabla^2 f(y)\|_{\op}\le L_H\|x-y\|$ for all $x,y\in\bR^d$.
\end{assumption}

\begin{lemma}\label{lem:lipschitz-taylor}
Under \Cref{ass:hessian}, for all $x,y\in\bR^d$, we have
\begin{align}
\|\nabla f(y)-\nabla f(x)-\nabla^2 f(x)(y-x)\|
&\le\tfrac{L_H}{2}\|y-x\|^2,\label{eq:gradient-taylor}\\
\big|f(y)-f(x)-\inner{\nabla f(x)}{y-x}
-\tfrac{1}{2}\inner{\nabla^2 f(x)(y-x)}{y-x}\big|
&\le\tfrac{L_H}{6}\|y-x\|^3.\label{eq:function-taylor}
\end{align}
\end{lemma}
The proofs can be found in Lemma 1.2.4 of \cite{Nesterov2018Lectures}. To develop an accelerated second-order method, we also require the third derivative to be Lipschitz continuous.

\begin{assumption}\label{ass:third}
The third derivative of the objective function $f$ is $L_T$-Lipschitz continuous, i.e., $\|\nabla^3 f(x)-\nabla^3 f(y)\|_{\op}\le L_T\|x-y\|$ for all $x,y\in\bR^d$.
\end{assumption}

\begin{lemma}\label{lem:lipschitz-interpolation}
Under \Cref{ass:third}, for all $x,y\in\bR^d$, we have
\begin{align}
\|\nabla^2 f(y)-\nabla^2 f(x)-\nabla^3 f(x)[y-x,\bullet,\bullet]\|_{\op}
&\le\tfrac{L_T}{2}\|y-x\|^2,\label{eq:taylor}\\
\big|f(y)-f(x)-\inner{\nabla f(x)}{y-x}
-\tfrac{1}{2}\inner{\nabla^2 f(x)(y-x)}{y-x}-\tfrac{1}{6}\nabla^3 f(x)[y-x]^{\otimes 3}\big|&\le\tfrac{L_T}{20}\|y-x\|^4.\label{eq:tensor-function-taylor}
\end{align}
\end{lemma}

The above lemma is classical in the complexity analysis of tensor methods, and its proof can be found in \cite{cartisgouldtoint2020second}. The following lemma provides tensor-free inequalities for the objective function value and the Hessian, which are key to designing accelerated second-order methods that attain, up to logarithmic factors, the same complexity guarantee as tensor methods without using tensors.

\begin{lemma}\label{lem:tensor-free}
Under \Cref{ass:third}, for all $x,y\in\bR^d$, we have
\begin{equation}\label{eq:tensor-free-function-taylor}
    \big|f(y)-f(x)-\tfrac{1}{2}\inner{\nabla f(x)+\nabla f(y)}{y-x}+\tfrac1{12}\inner{(\nabla^2 f(y)-\nabla^2 f(x))(y-x)}{y-x}\big|\le \tfrac{L_T}{192}\|y-x\|^4.
\end{equation}
Moreover, for $0\le \tau\le 1$, we also have
\begin{equation}\label{eq:chord} 
\|\nabla^2 f((1-\tau)x+\tau y)-(1-\tau)\nabla^2 f(x)-\tau\nabla^2 f(y)\|_{\op}\le\tfrac{L_T}{2}\tau(1-\tau)\|y-x\|^2.
\end{equation}
\end{lemma}
The above inequalities are classical results in numerical analysis; see, for example, \cite{lampret2004invitation,bui2026elementary}. To the best of our knowledge, using these inequalities to analyze the convergence of optimization methods appears to be new. For completeness, we provide the proofs in \Cref{app:hermite}.

\paragraph{Organization.} The remainder of this paper is organized as follows. In \Cref{sec:framework}, we introduce the proximal point Newton framework and show that a single-loop method achieves the same complexity guarantee as the cubic regularized Newton method. Building on this framework, \Cref{sec:accelerated} develops an
accelerated proximal point Newton method that exploits third-derivative smoothness using only gradients and Hessians, yielding improved complexity guarantees. In \Cref{sec:binary-cg}, we further develop an adaptive Newton--CG method based on capped CG and the Lanczos method,
which requires no prior knowledge of Lipschitz constants, and establish its Hessian--vector product complexity. Finally, \Cref{sec:numerical} presents numerical experiments comparing our method with existing second-order methods.
 
\section{The Proximal Point Newton Framework}\label{sec:framework}
In this section, we first provide insight into the design of our proximal point Newton framework. As a direct application of the framework, we then show that a single-loop proximal point Newton method achieves the same complexity guarantee as the cubic regularized Newton method under Hessian Lipschitz continuity.

\subsection{The framework design}
Since the objective function in \eqref{eq:problem} is potentially nonconvex, an unregularized Newton step need not be well defined due to indefiniteness or singularity. Thus, at iteration $k$, we consider the following proximal point subproblem:
\begin{equation*}
    \min_{x\in\bR^d} \ \Phi_k(x) \coloneqq f(x)+\mu\|x-x^k\|^2,
\end{equation*}
where $\mu > 0$ is a predefined regularization parameter. The function $\Phi_k(\cdot)$ is referred to as the \textit{proximal function}. Since $f$ is bounded below, the proximal function is coercive and admits a minimizer. Nevertheless, $\Phi_k$ may remain nonconvex, and the quadratic regularization alone does not guarantee that its Newton system is well defined at every iterate. Therefore, we need to perform a \textit{negative-curvature detection} before computing the Newton direction of the proximal function. We split this into two branches: if the negative-curvature detection returns sufficient negative curvature of $f(\cdot)$, then we employ the negative-curvature direction as our descent direction; otherwise, $\nabla^2\Phi_k$ is positive definite, and the Newton direction of the proximal function is used as our update direction. For the proximal function, we have
\begin{equation*}
\textstyle \nabla\Phi_k(x) =\nabla f(x)+2\mu(x-x^k), \ \text{and} \ \nabla^2\Phi_k(x)=\nabla^2 f(x)+2\mu I.
\end{equation*}
Whenever $\nabla^2\Phi_k(x^k)\succ0$, we define the local norm as
\begin{equation*}
    \|u\|_{k,\mu}^2 \coloneqq \inner{\nabla^2\Phi_k(x^k)u}{u}=\inner{[\nabla^2 f(x^k)+2\mu I]u}{u}, \ \forall u \in \bR^d.
\end{equation*}

As we will show in our analysis, using a negative-curvature direction is a favorable case, as it always yields a sufficient decrease in the objective function value $f(\cdot)$. The main analysis therefore focuses on the Newton branch. In this case, we seek to minimize the proximal function, and hence an inner loop is required. For clarity, the outer iteration index is $k\ge0$, and the inner iteration index is $t\ge0$, with $x^{k,0} \coloneqq x^k$. For a fixed outer iterate $x^k$, all inner iterates $x^{k,t}$ are associated with the same proximal function $\Phi_k$. A sufficient stopping criterion for an inner iterate $x^{k,t}$ is
\begin{equation*}
\textstyle \|\nabla\Phi_k(x^{k,t})\|\le\frac{\epsg}{2}, \ 2\mu\|x^{k,t}-x^k\|\le\frac{\epsg}{2}, \ \lambda_{\min}(\nabla^2 f(x^{k,t})) > -\epsH.
\end{equation*}
The first two inequalities control the proximal residual and its bias, respectively. Let $\Delta > 0$ be a fixed radius. We apply the fixed-radius analysis strategy in the Newton branch \citep{ye2005second,curtis2018concise,zhang2026homogeneous}. If the magnitude of the Newton direction is larger than the radius $\Delta$, which we refer to as the \textit{large-step} case, we can show that the function value decreases sufficiently on the order of $\Delta^2$. Otherwise, we are in the \textit{small-step} case, where a full Newton step is taken. We can then show that $\|\nabla\Phi_k(x^{k,t})\|$ decreases geometrically.

\subsection{Single-loop proximal point Newton method}
In this subsection, we assume only that \Cref{ass:hessian} holds. Surprisingly, we show that the proximal point subproblem requires at most one Newton update, and hence the proximal point framework motivates a single-loop regularized Newton method. Fix an outer iteration $k$. We first perform negative-curvature detection at $x^k$. If $\lambda_{\min}(\nabla^2 f(x^k)) \le -\mu$, then we choose the negative-curvature direction as
\begin{equation*}
    \textstyle v^{k}\in\argmin_{\|v\|=1}\inner{\nabla^2 f(x^k)v}{v}.
\end{equation*}
Define the cubic regularized Newton model at iterate $x^k$ by
\begin{equation*}
    \textstyle M_k(x) \coloneqq \inner{\nabla f(x^k)}{x-x^k}+\tfrac{1}{2}\inner{\nabla^2 f(x^k)(x-x^k)}{x-x^k}+\tfrac{L_H}{6}\|x-x^k\|^3.
\end{equation*}
We then select the next iterate by comparing the third-order subproblem values at the candidate points $\{x^k+\mu v^k/L_H, x^k-\mu v^k/L_H\}$. Specifically, the next iterate $x^{k+1}$ is chosen as
\begin{equation}\label{eq:negative-curvature-next iterate}
    \textstyle x^{k+1} \in \argmin_{x \in \{x^k \pm \mu v^k/L_H\}} M_k(x)
\end{equation}
such that the third-order subproblem $M_k(\cdot)$ attains the smallest value. If $\lambda_{\min}(\nabla^2 f(x^k)) > -\mu$, then $\nabla^2 \Phi_k(x^k)$ is positive definite. We compute the Newton direction
\begin{equation}\label{eq:newton-direction}
    \textstyle d^k \coloneqq \nabla^2\Phi_k(x^k)^{-1}\nabla\Phi_k(x^k).
\end{equation}
If $\anorm{d^k}\ge\Delta$, we are in the \textit{large-step case}: set $x^{k+1}=x^k-\Delta d^k/\anorm{d^k}$, so that $\anorm{x^{k+1}-x^k}=\Delta$. We can show that the function value decreases sufficiently. Otherwise, we are in the \textit{small-step case}: take the full Newton step $x^{k+1}=x^k-d^k$, and we can show that the next iterate is already an $(\epsg,\epsH)$-approximate second-order stationary point. The complete algorithm is presented in \Cref{alg:ppnewton}.

\begin{algorithm}[h]
\caption{Proximal Point Newton Method}\label{alg:ppnewton}
\begin{algorithmic}[1]
\Require Hessian Lipschitz constant $L_H>0$, accuracies $\epsg>0$ and $\epsH>0$.
\State \textbf{Initialization:} initial point $x^0\in\bR^d$, regularization parameter $\mu > 0$, radius $\Delta > 0$.
\For{$k=0,1,2,\ldots$}
  \If{$\lambda_{\min}(\nabla^2 f(x^k)) \le -\mu$}
    \State Choose the negative-curvature direction $v^{k}\in\argmin_{\|v\|=1}\inner{\nabla^2 f(x^k)v}{v}$.
    \State Update the iterate $x^{k+1}$ by \eqref{eq:negative-curvature-next iterate}.
  \Else
    \State Define the proximal function $\Phi_k(\cdot)$ at iterate $x^k$.
    \State Compute the Newton direction $d^k$ by \eqref{eq:newton-direction}.
    \State Update the iterate $x^{k+1}$
    \If{$\anorm{d^k}\ge\Delta$}
      \State Update $x^{k+1}=x^k-\Delta d^k/\anorm{d^k}$.
    \Else
      \State Update $x^{k+1}=x^k-d^k$.
    \EndIf
  \EndIf
\EndFor
\end{algorithmic}
\end{algorithm}

\begin{lemma}\label{lem:negative-curvature-decrease}
Under \Cref{ass:hessian}, if $\lambda_{\min}(\nabla^2f(x^k))\le-\mu$, then we have
\begin{align*}
  f(x^{k+1})-f(x^k)\le-\frac{\mu^3}{3L_H^2}. 
\end{align*}
\end{lemma}
\begin{proof}
By the update $x^{k+1} \in \argmin_{x \in \{x^k+\mu v^k/L_H, x^k-\mu v^k/L_H\}} M_k(x)$, we have
\begin{align*}
    \textstyle M_k(x^{k+1}) \le \min \big\{M_k(x^k+\mu v^k/L_H), M_k(x^k-\mu v^k/L_H)\big\} \le \frac{1}{2}\big(M_k(x^k+\mu v^k/L_H)+M_k(x^k-\mu v^k/L_H)\big),
\end{align*}
which implies that
\begin{align*}
    \textstyle M_k(x^{k+1}) \le \frac{\mu^2}{2L_H^2}\inner{\nabla^2f(x^k)v^k}{v^k}+\frac{\mu^3}{6L_H^2} \le -\frac{\mu^3}{3L_H^2},
\end{align*}
where the last inequality follows from $\inner{\nabla^2f(x^k)v^k}{v^k}=\lambda_{\min}(\nabla^2f(x^k))\le-\mu$. By \eqref{eq:function-taylor}, we then have
\begin{align*}
    \textstyle f(x^{k+1}) \le f(x^k) + M_k(x^{k+1}) \le f(x^k) - \frac{\mu^3}{3L_H^2},
\end{align*}
which completes the proof.
\end{proof}

\begin{lemma}[Large-step case]\label{lem:newton-decrease}
Under \Cref{ass:hessian}, suppose $\lambda_{\min}(\nabla^2f(x^k))>-\mu$ and $16L_H\Delta\le\mu^{3/2}$. In the large-step case $\anorm{d^k}\ge\Delta$, we have
\begin{equation*}
    \textstyle f(x^{k+1})-f(x^k)\le-\frac{\Delta^2}{2}.
\end{equation*}
\end{lemma}
\begin{proof}
By applying \eqref{eq:function-taylor} and the definition of the proximal function, we have
\begin{align*}
    \textstyle f(x^{k+1})-f(x^k)\le \inner{\nabla\Phi_k(x^k)}{x^{k+1}-x^k}+\frac12\anorm{x^{k+1}-x^k}^2-\mu\|x^{k+1}-x^k\|^2+\frac{L_H}{6}\|x^{k+1}-x^k\|^3.
\end{align*}
Since $\nabla^2\Phi_k(x^k)\succ\mu I$, it follows that
\begin{align*}
    \textstyle \sqrt\mu\|x^{k+1}-x^k\|\le\anorm{x^{k+1}-x^k}
    =\min\{\anorm{d^k},\Delta\}\le\Delta.
\end{align*}
The Newton equation gives
\begin{align*}
\textstyle \inner{\nabla\Phi_k(x^k)}{x^{k+1}-x^k}=-\frac{\Delta}{\|d^k\|_{k,\mu}}\inner{\nabla^2\Phi_k(x^k)d^k}{d^k} =-\Delta\|d^k\|_{k,\mu} \le -\Delta^2= -\|x^{k+1}-x^k\|_{k,\mu}^2.
\end{align*}
which implies that
\begin{equation*}
    \textstyle f(x^{k+1})-f(x^k)\le-\frac12\anorm{x^{k+1}-x^k}^2 -\big(\mu-\frac{L_H\Delta}{6\sqrt\mu}\big)\|x^{k+1}-x^k\|^2.
\end{equation*}
The second term is nonnegative under the condition $16L_H\Delta\le\mu^{3/2}$. Since $\anorm{d^k}\ge\Delta$, we have $\anorm{x^{k+1}-x^k}=\Delta$, which completes the proof.
\end{proof}

\begin{lemma}[Small-step case]\label{lem:baseline-small-step}
Under \Cref{ass:hessian}, suppose $\lambda_{\min}(\nabla^2f(x^k))>-\mu$. In the small-step case $\anorm{d^k}<\Delta$, the full step $x^{k+1}=x^k-d^k$ satisfies
\begin{equation*}
    \textstyle \|\nabla f(x^{k+1})\|\le\frac{L_H\Delta^2}{2\mu}+2\sqrt\mu\Delta \ \text{and} \ \nabla^2f(x^{k+1})\succ-\big(\mu+\frac{L_H\Delta}{\sqrt\mu}\big)I.
\end{equation*}
\end{lemma}
\begin{proof}
First, note that $\|\nabla f(x^{k+1})\|\le\|\nabla\Phi_k(x^{k+1})\|+2\mu\|x^{k+1}-x^k\|$. Since $\nabla^2\Phi_k(x^k)\succ\mu I$ and we take the full Newton step $x^{k+1}=x^k-d^k$, we have
\begin{align*}
    \textstyle \sqrt\mu\|x^{k+1}-x^k\|\le\anorm{x^{k+1}-x^k}<\Delta \ \text{and} \ 
    \nabla\Phi_k(x^k)+\nabla^2\Phi_k(x^k)(x^{k+1}-x^k)=0.
\end{align*}
Consequently, \eqref{eq:gradient-taylor} gives
\begin{align*}
    \textstyle \|\nabla\Phi_k(x^{k+1})\|\le\frac{L_H}{2}\|x^{k+1}-x^k\|^2
    \le\frac{L_H\Delta^2}{2\mu},
\end{align*}
which implies that $\|\nabla f(x^{k+1})\|\le \frac{L_H\Delta^2}{2\mu}+2\sqrt\mu\Delta$.
Finally, Hessian Lipschitz continuity gives
\begin{align*}
    \textstyle \nabla^2f(x^{k+1})\succeq\nabla^2f(x^k)-L_H\|x^{k+1}-x^k\|I
    \succ-\big(\mu+\frac{L_H\Delta}{\sqrt\mu}\big)I,
\end{align*}
which completes the proof.
\end{proof}

\begin{theorem}\label{thm:baseline-complexity}
Under \Cref{ass:hessian}, choose $\mu=\min\{\sqrt{L_H\epsg},\epsH/2\}$ and $\Delta=\frac{\mu^{3/2}}{16L_H}$. Then \Cref{alg:ppnewton} returns an $(\epsg,\epsH)$-approximate second-order stationary point using at most
\begin{align*}
    \cO\big((f(x^0)-f_{\inf})\max\{\sqrt{L_H}\epsg^{-3/2},L_H^2\epsH^{-3}\}\big)
\end{align*}
second-order oracle evaluations. 
\end{theorem}
\begin{proof}
Under the choices of $\mu$ and $\Delta$, the negative-curvature direction update yields a decrease of at least $\Delta^2/2$ by \Cref{lem:negative-curvature-decrease}. \Cref{lem:newton-decrease} gives a decrease of at least $\Delta^2/2$ for the large-step Newton update. Otherwise, by \Cref{lem:baseline-small-step} and the choices of $\mu$ and $\Delta$, a full Newton update yields an $(\epsg,\epsH)$-approximate second-order stationary point after one update. After $K$ iterations, we have
\begin{align*}
    \textstyle \frac{K\Delta^2}{2}
    \le \sum_{k=0}^{K-1}(f(x^k)-f(x^{k+1}))
    = f(x^0)-f(x^K)
    \le f(x^0)-f_{\inf}.
\end{align*}
Therefore, after at most $\lfloor 2(f(x^0)-f_{\inf})/\Delta^2\rfloor$ iterations, an $(\epsg,\epsH)$-approximate second-order stationary point can be returned. The proof is complete.
\end{proof}

\section{Accelerated Proximal Point Newton Method}\label{sec:accelerated}
In this section, we assume that both \Cref{ass:hessian} and \Cref{ass:third} hold. We follow a construction similar to that of the single-loop proximal point Newton method: First detect negative curvature, and then either select a negative-curvature direction or take Newton steps for the proximal function. By exploiting the additional Lipschitz continuity of the third derivative, the fixed radius $\Delta$ can be chosen more aggressively, leading to a larger decrease in the function value and hence requiring fewer second-order oracle evaluations. From an algorithmic perspective, the main difference is that the algorithm in this section is no longer single-loop. We again separate the large-step and small-step cases when taking Newton steps: the former yields a sufficient decrease in the function value, while the latter yields a geometric decrease until an $(\epsg,\epsH)$-approximate second-order stationary point is returned or negative curvature is detected.

\subsection{Algorithm design}
Fix an outer iteration $k$. For negative-curvature detection, since an inner loop is required, merely checking $\lambda_{\min}(\nabla^2 f(x^k))>-\mu$ guarantees only that $\nabla^2 \Phi_k(\cdot)$ is positive definite at the outer iterate $x^k$. We also need to control the change in the proximal Hessian along the inner iterates. Since the third derivative of $f$ is $L_T$-Lipschitz continuous, we have
\begin{align*}
    \textstyle \|\nabla^2 f(x^{k,t})-\nabla^2 f(x^k)-\nabla^3 f(x^k)[x^{k,t}-x^k, \bullet, \bullet]\|_{\op} \le \frac{L_T}{2}\|x^{k,t}-x^k\|^2.
\end{align*}
Ideally, one could directly use the directional third-order curvature $\nabla^3 f(x^k)[x^{k,t}-x^k,\bullet,\bullet]$ to detect whether the Hessian changes too much along the inner loop, but such a test would require explicit access to the third derivative. To retain a second-order method, we instead use the Hessian difference $\nabla^2 f(x^{k,t})-\nabla^2 f(x^k)$, which motivates the following \textit{relative negative-curvature detection}
\begin{equation}\label{eq:gate}
    \min_{s\in\{-1,1\}}\lambda_{\min} \big(\nabla^2 f(x^k)+\mu I+4s(\nabla^2 f(x^{k,t})-\nabla^2 f(x^k))\big)\le0
\end{equation}
at an inner iterate $x^{k,t}$. At $t=0$, this condition is exactly the negative-curvature condition in \Cref{alg:ppnewton}. If \eqref{eq:gate} holds, choose the sign and direction
\begin{equation}\label{eq:certificate}
(s^{k,t},v^{k,t})\in\argmin_{s\in\{-1,1\}, \|v\|=1}
\inner{\big(\nabla^2 f(x^k)+\mu I+4s(\nabla^2 f(x^{k,t})-\nabla^2 f(x^k))\big)v}{v}.
\end{equation}
Define the tensor-free model at iterate $x^k$ by
\begin{equation*}
    \textstyle \Omega_k(x) \coloneqq \tfrac{1}{2}\inner{\nabla f(x^k)+\nabla f(x)}{x-x^k}-\tfrac1{12}\inner{(\nabla^2 f(x)-\nabla^2 f(x^k))(x-x^k)}{x-x^k}+\tfrac{L_T}{192}\|x-x^k\|^4.
\end{equation*}
By \eqref{eq:tensor-free-function-taylor}, we have $f(x)-f(x^k)\le \Omega_k(x)$. Select the next iterate by
\begin{equation}\label{eq:comparison}
 x^{k+1}\in
 \begin{cases}
    \argmin_{x\in\{x^k\pm\sqrt{\mu/L_T} v^{k,0}\}} \Omega_k(x),&t=0,\\
    \argmin_{\substack{x\in\{x^k-4s^{k,t}(x^{k,t}-x^k),x^k+4s^{k,t}(x^{k,t}-x^k)\pm\sqrt{\mu/L_T} v^{k,t}\}}}\Omega_k(x),&t\ge1.
 \end{cases}
\end{equation}
The above choice yields a sufficient decrease in the function value at the next iterate $x^{k+1}$.

If \eqref{eq:gate} does not hold, both choices $s=\pm1$ yield positive definite matrices, and it follows that
\begin{equation*}
    \textstyle -\frac{1}{4}\big(\nabla^2 f(x^k)+\mu I\big) \prec \nabla^2 f(x^{k,t})-\nabla^2 f(x^k) \prec \frac{1}{4}\big(\nabla^2 f(x^k)+\mu I\big).
\end{equation*}
Rearranging these inequalities gives
\begin{equation}\label{eq:relative-reserve}
    \textstyle \tfrac34\nabla^2\Phi_k(x^k)+\tfrac\mu4 I
    \prec\nabla^2\Phi_k(x^{k,t})
    \prec\tfrac54\nabla^2\Phi_k(x^k)-\tfrac\mu4 I,
\end{equation}
In particular, we have $\nabla^2\Phi_k(x^{k,t})\succ\mu I$ and $\nabla^2f(x^{k,t})\succ-\mu I$. We also have
\begin{equation}\label{eq:variation-norm-bound}
\begin{aligned}
    &\textstyle\|\nabla^2\Phi_k(x^k)^{-1/2}(\nabla^2\Phi_k(x^{k,t})-\nabla^2\Phi_k(x^k))\nabla^2\Phi_k(x^k)^{-1/2}\|_{\op}\le\tfrac14,\\
    &\textstyle\|\nabla^2\Phi_k(x^k)^{1/2}\nabla^2\Phi_k(x^{k,t})^{-1}\nabla^2\Phi_k(x^k)^{1/2}\|_{\op}\le\tfrac43.
\end{aligned}
\end{equation}
Therefore, the Newton direction at the inner iterate $x^{k,t}$ is well defined. We compute the Newton direction and the stepsize
\begin{equation}\label{eq:accelerated-direction}
    \textstyle d^{k,t}\coloneqq\nabla^2\Phi_k(x^{k,t})^{-1}\nabla\Phi_k(x^{k,t}), \ \eta_{k,t}\coloneqq\min\big\{1,\frac{\Delta}{\anorm{d^{k,t}}}\big\},
\end{equation}
and set $x^{k,t+1}=x^{k,t}-\eta_{k,t}d^{k,t}$. Similarly, if $\anorm{d^{k,0}}\ge\Delta$, i.e., the first Newton update reaches the boundary of the fixed radius, we are in the large-step case. We can then show that the function value decreases sufficiently at $x^{k,1}$ and set $x^{k+1}=x^{k,1}$. Otherwise, we are in the small-step case, with $\anorm{d^{k,0}}<\Delta$, and we take the full Newton step. We will show that all subsequent Newton steps in this inner loop remain small, so that full Newton steps are taken and geometric decrease is achieved until convergence or negative curvature is detected. The complete algorithm is presented in \Cref{alg:accppnewton}.

\begin{algorithm}[h]
\caption{Accelerated Proximal Point Newton Method}\label{alg:accppnewton}
\begin{algorithmic}[1]
\Require third derivative Lipschitz constant $L_T>0$, accuracies $\epsg>0$ and $\epsH>0$.
\State \textbf{Initialization:} initial point $x^0\in\bR^d$, regularization parameter $\mu > 0$, radius $\Delta > 0$.
\For{$k=0,1,2,\ldots$}
  \State Define the proximal function $\Phi_k(\cdot)$ at iterate $x^k$.
  \State Initialize $x^{k,0}=x^k$.
  \For{$t=0,1,2,\ldots$}
    \If{$\min_{s\in\{-1,1\}}\lambda_{\min}(\nabla^2 f(x^k)+\mu I+4s(\nabla^2 f(x^{k,t})-\nabla^2 f(x^k)))\le0$}
      \State Choose the negative-curvature sign and direction $(s^{k,t},v^{k,t})$ by \eqref{eq:certificate}.
      \State Update the iterate $x^{k+1}$ by \eqref{eq:comparison}; \textbf{break}.
    \Else
      \If{$t=1$ and $\anorm{x^{k,1}-x^k}=\Delta$}
        \State Update the iterate $x^{k+1}=x^{k,1}$; \textbf{break}.
      \Else
        \State Compute the Newton direction $d^{k,t}$ and the stepsize $\eta_{k,t}$ by \eqref{eq:accelerated-direction}.
        \State Update the iterate $x^{k,t+1}=x^{k,t}-\eta_{k,t}d^{k,t}$.
      \EndIf
    \EndIf
  \EndFor
\EndFor
\end{algorithmic}
\end{algorithm}

\subsection{Convergence analysis}
The following two lemmas show that the negative-curvature direction yields a sufficient decrease in the function value either at the outer iterate $x^k$ or at any inner iterate, as long as the condition \eqref{eq:gate} is satisfied.
\begin{lemma}\label{lem:decrease-initial}
Under \Cref{ass:third}, if \eqref{eq:gate} holds at $x^{k,0}=x^k$, then we have
\begin{equation*}
   f(x^{k+1})-f(x^k)\le-\frac{9\mu^2}{20L_T}. 
\end{equation*}
\end{lemma}
\begin{proof}
Note that $x^{k+1} \in \argmin_{x\in{x^k\pm\sqrt{\mu/L_T} v^{k,0}}} \Omega_k(x)$ when \eqref{eq:gate} holds at $x^{k,0}=x^k$. We then have
\begin{align*}
    \textstyle \Omega_k(x^{k+1}) \le \frac{1}{2}(\Omega_k(x^k+\sqrt{\mu/L_T} v^{k,0}) + \Omega_k(x^k-\sqrt{\mu/L_T} v^{k,0})) \le \frac{\mu}{2L_T}\inner{\nabla^2f(x^k)v^{k,0}}{v^{k,0}}+\frac{\mu^2}{20L_T}.
\end{align*}
Since, \eqref{eq:gate} reduces to $\lambda_{\min}(\nabla^2 f(x^k) + \mu I) \le 0$ at $t=0$, it follows that $\Omega_k(x^{k+1}) \le -\frac{9\mu^2}{20L_T}$. Therefore, using $f(x^{k+1})-f(x^k)\le \Omega_k(x^{k+1})$ completes the proof.
\end{proof}

\begin{lemma}\label{lem:decrease}
Under \Cref{ass:third}, suppose that $256L_T\Delta^2\le\mu^2$. If \eqref{eq:gate} holds at an inner iterate $x^{k,t}$ with $t\ge1$ and $\anorm{x^{k,t}-x^k}\le2\Delta$, then we have
\begin{align*}
   f(x^{k+1})-f(x^k)\le-\frac{\mu^2}{32L_T}. 
\end{align*}
\end{lemma}
\begin{proof}
Since \eqref{eq:gate} holds at an inner iterate $x^{k,t}$ with $t\ge1$, by \eqref{eq:certificate}, we have
\begin{align*}
    \textstyle \inner{\big(\nabla^2 f(x^k)+\mu I+4s^{k,t}(\nabla^2 f(x^{k,t})-\nabla^2 f(x^k))\big)v^{k,t}}{v^{k,t}} \le 0.
\end{align*}
By using \eqref{eq:taylor} and $\|v^{k,t}\|=1$, we further obtain
\begin{align*}
    \textstyle \inner{\nabla^2 f(x^k)v^{k,t}}{v^{k,t}}+4s^{k,t}\nabla^3 f(x^k)[x^{k,t}-x^k,v^{k,t},v^{k,t}] \le -\mu + 2L_T\|x^{k,t}-x^k\|^2.
\end{align*}
By the definition of $x^{k+1}$ in \eqref{eq:negative-curvature-next iterate} at inner iteration $t \ge 1$, we have
\begin{align*}
    &\textstyle \Omega_k(x^{k+1}) \\
    \le \ &\textstyle \frac{1}{2}\Omega_k(x^k-4s^{k,t}(x^{k,t}-x^k)) \\
    &\textstyle +\frac{1}{4}\Omega_k(x^k+4s^{k,t}(x^{k,t}-x^k)+\sqrt{\mu/L_T} v^{k,t})+\frac{1}{4}\Omega_k(x^k+4s^{k,t}(x^{k,t}-x^k)-\sqrt{\mu/L_T} v^{k,t}) \\
    = \ &\textstyle 8\inner{\nabla^2 f(x^k)(x^{k,t}-x^k)}{x^{k,t}-x^k}+\frac{\mu}{4L_T}(\inner{\nabla^2 f(x^k)v^{k,t}}{v^{k,t}}+4s^{k,t}\nabla^3 f(x^k)
    [x^{k,t}-x^k,v^{k,t},v^{k,t}])\\
    &\textstyle +\frac{L_T}{40}\|-4s^{k,t}(x^{k,t}-x^k)\|^4 + \frac{L_T}{80}\|4s^{k,t}(x^{k,t}-x^k)
    +\sqrt{\mu/L_T} v^{k,t}\|^4 + \frac{L_T}{80}\|4s^{k,t}(x^{k,t}-x^k)-\sqrt{\mu/L_T} v^{k,t}\|^4 \\
    \le \ &\textstyle 8\inner{\nabla^2 f(x^k)(x^{k,t}-x^k)}{x^{k,t}-x^k}+\frac{\mu}{4L_T}(-\mu + 2L_T\|x^{k,t}-x^k\|^2)\\
    &\textstyle +\frac{L_T}{40}\|-4s^{k,t}(x^{k,t}-x^k)\|^4 + \frac{L_T}{80}\|4s^{k,t}(x^{k,t}-x^k)
    +\sqrt{\mu/L_T} v^{k,t}\|^4 + \frac{L_T}{80}\|4s^{k,t}(x^{k,t}-x^k)-\sqrt{\mu/L_T} v^{k,t}\|^4.
\end{align*}
On the one hand, since $\nabla^2 f(x^k)=\nabla^2\Phi_k(x^k)-2\mu I$, we have
\begin{align*}
    \textstyle \inner{\nabla^2 f(x^k)(x^{k,t}-x^k)}{x^{k,t}-x^k}=\anorm{x^{k,t}-x^k}^2 - 2\|x^{k,t}-x^k\|^2. 
\end{align*}
On the other hand, direct calculation gives
\begin{align*}
    &\textstyle \frac{1}{2}\|-4s^{k,t}(x^{k,t}-x^k)\|^4
    +\frac{1}{4}\|4s^{k,t}(x^{k,t}-x^k)
    +\sqrt{\mu/L_T} v^{k,t}\|^4 + \frac{1}{4}\|4s^{k,t}(x^{k,t}-x^k)-\sqrt{\mu/L_T} v^{k,t}\|^4\\
    = \ &\textstyle 256\|x^{k,t}-x^k\|^4+\frac{16\mu}{L_T}\|x^{k,t}-x^k\|^2+\frac{32\mu}{L_T}\inner{x^{k,t}-x^k}{v^{k,t}}^2+\frac{\mu^2}{2L_T^2}\\
    \le \ &\textstyle 256\|x^{k,t}-x^k\|^4+\frac{48\mu}{L_T}\|x^{k,t}-x^k\|^2+\frac{\mu^2}{2L_T^2},
\end{align*}
where the last step uses the Cauchy--Schwarz inequality and $\|v^{k,t}\|=1$. Therefore, combining the above inequalities gives
\begin{align*}
    \textstyle f(x^{k+1}) - f(x^k) \le \Omega_k(x^{k+1}) \le 8\anorm{x^{k,t}-x^k}^2-\frac{131\mu}{10}\|x^{k,t}-x^k\|^2+\frac{64L_T}{5}\|x^{k,t}-x^k\|^4-\frac{9\mu^2}{40L_T}.
\end{align*}
Note that $\nabla^2 \Phi_k(x^k) \succ \mu I$ holds because \Cref{alg:accppnewton} reaches the inner loop with $t \ge 1$, which implies that condition \eqref{eq:gate} is not satisfied at $x^k$. Under the condition $\|x^{k,t}-x^k\|_{k,\mu} \le 2\Delta$ and $256L_T\Delta^2 \le \mu^2$, we have
\begin{align*}
    \textstyle \frac{64L_T}{5}\|x^{k,t}-x^k\|^4
    \le\frac{256L_T\Delta^2}{5\mu}\|x^{k,t}-x^k\|^2
    \le\frac\mu5\|x^{k,t}-x^k\|^2.
\end{align*}
which yields
\begin{align*}
   \textstyle f(x^{k+1})-f(x^k)\le32\Delta^2-\frac{9\mu^2}{40L_T}
    \le-\frac{\mu^2}{10L_T}\le-\frac{\mu^2}{32L_T}.
\end{align*}
The proof is complete.
\end{proof}
\begin{remark}
    Note that the above analysis requires the condition $\|x^{k,t}-x^k\|_{k,\mu} \le 2\Delta$. As we will show later, this condition is automatically satisfied in both the large-step case and the small-step case (\Cref{lem:initial-newton} and \Cref{lem:contraction}).
\end{remark}

\begin{lemma}[Large-step case]\label{lem:initial-newton}
Under \Cref{ass:hessian,ass:third}, suppose $256L_T\Delta^2\le\mu^2$ and the first Newton step satisfies $\anorm{d^{k,0}}\ge\Delta$. Whether \eqref{eq:gate} holds or does not hold at $x^{k,1}$, we have
\begin{equation*}
    f(x^{k+1})-f(x^k)\le-\frac{2}{5}\Delta^2.
\end{equation*}
\end{lemma}
\begin{proof}
If \eqref{eq:gate} holds at $x^{k,1}$, since $\|x^{k,1}-x^k\|_{k,\mu}=\Delta \le 2\Delta$, \Cref{lem:decrease} gives a decrease of at least $\mu^2/(32L_T)\ge\frac{2}{5}\Delta^2$ under the choice $256L_T\Delta^2\le\mu^2$. It remains to consider the case in which \eqref{eq:gate} does not hold at $x^{k,1}$. By $f(x)-f(x^k) = \Phi_k(x)-\Phi_k(x^k)-\mu\|x-x^k\|^2$, we have
\begin{align*}
    &\textstyle f(x^{k,1})-f(x^k) \\
    = \ &\textstyle \Phi_k(x^{k,1}) - \Phi_k(x^k) - \mu\|x^{k,1}-x^k\|^2 \\
    = \ &\textstyle \inner{\nabla\Phi_k(x^k)}{x^{k,1}-x^k}
    +\int_0^1(1-\tau)\inner{\nabla^2\Phi_k\bigl(x^k+\tau(x^{k,1}-x^k)\bigr)(x^{k,1}-x^k)}{x^{k,1}-x^k} \mathrm{d}\tau
    -\mu\|x^{k,1}-x^k\|^2\\
    \le \ &\textstyle \inner{\nabla\Phi_k(x^k)}{x^{k,1}-x^k}
    +\frac{1}{3}\inner{\nabla^2\Phi_k(x^k)(x^{k,1}-x^k)}{x^{k,1}-x^k}\\
    &\textstyle+\frac{1}{6}\inner{\nabla^2\Phi_k(x^{k,1})(x^{k,1}-x^k)}{x^{k,1}-x^k}-\mu\|x^{k,1}-x^k\|^2+\frac{L_T}{24}\|x^{k,1}-x^k\|^4,
\end{align*}
where the last step follows from \eqref{eq:chord}. Moreover, by the Newton equation $\nabla^2 \Phi_k(x^k)d^{k,0}=\nabla \Phi_k(x^k)$ and $x^{k,1} = x^k-\eta_{k,0} d^{k,0}$, it follows that
\begin{align*}
    \textstyle \inner{\nabla\Phi_k(x^k)}{x^{k,1}-x^k}
    =-\frac{1}{\eta_{k,0}}\anorm{x^{k,1}-x^k}^2
    \le-\anorm{x^{k,1}-x^k}^2,
\end{align*}
where we use $\eta_{k,0}\le1$ in the last inequality. Therefore, by \eqref{eq:relative-reserve}, we obtain
\begin{align*}
    &\textstyle f(x^{k,1})-f(x^k) \\
    \le &\textstyle
    -\frac{2}{3}\anorm{x^{k,1}-x^k}^2
    +\frac{1}{6}\inner{\nabla^2\Phi_k(x^{k,1})(x^{k,1}-x^k)}{x^{k,1}-x^k}
    -\mu\|x^{k,1}-x^k\|^2
    +\frac{L_T}{24}\|x^{k,1}-x^k\|^4\\
    \le &\textstyle
    -\frac{11}{24}\anorm{x^{k,1}-x^k}^2
    -\mu\|x^{k,1}-x^k\|^2
    +\frac{L_T}{24}\|x^{k,1}-x^k\|^4.
\end{align*}
Recall that $\anorm{x^{k,1}-x^k}^2 = \Delta$ and $\nabla^2 \Phi_k(x^k) \succ \mu I$. We thus have $\|x^{k,1}-x^k\|^2\le\Delta^2/\mu$. Under the choice $256L_T\Delta^2\le\mu^2$, it follows that $f(x^{k+1}) = f(x^{k,1}) \le f(x^k) -\frac{2}{5}\Delta^2$, which completes the proof.
\end{proof}

In the following, we focus on the small-step case. We show that as long as $\anorm{d^{k,0}}<\Delta$, every subsequent inner step takes a full Newton update, i.e., $\anorm{d^{k,t}}<\Delta$ for all $t \ge 1$.
\begin{lemma}[Small-step case: contraction]\label{lem:contraction}
Under \Cref{ass:hessian,ass:third}, suppose $256L_T\Delta^2\le\mu^2$ and the first Newton step satisfies $\anorm{d^{k,0}}<\Delta$. Then every inner iterate $x^{k,t}$ with $t\ge1$ satisfies
\begin{equation}\label{eq:contraction}
\textstyle \anorm{x^{k,t}-x^{k,t-1}}\le2^{1-t}\Delta, \ 
    \anorm{x^{k,t}-x^k}\le2\Delta, \ 
    \|x^{k,t}-x^k\|^2\le\frac{4\Delta^2}{\mu}.
\end{equation}
\end{lemma}
\begin{proof}
We first prove by induction that every inner iterate $x^{k,t}$ with $t\ge1$ satisfies
\begin{align*}
    \textstyle \eta_{k,t-1}=1, \ \anorm{x^{k,t}-x^{k,t-1}}\le2^{1-t}\Delta.
\end{align*}
The case $t=1$ is trivial since $\anorm{x^{k,1}-x^{k,0}}=\|d^{k,0}\|_{k,\mu} < \Delta$. Suppose the claim holds at index $t\ge1$ and $x^{k,t+1}$ is reached. By the definition $d^{k,t}$, it holds that
\begin{align*}
    \textstyle \anorm{d^{k,t}} = \anorm{\nabla^2 \Phi_k(x^{k,t})^{-1}\nabla \Phi_k(x^{k,t})} = \|\nabla^2\Phi_k(x^k)^{1/2}\nabla^2\Phi_k(x^{k,t})^{-1}\nabla\Phi_k(x^{k,t})\|.
\end{align*}
Since
\begin{align*}
   &\textstyle \nabla^2\Phi_k(x^k)^{1/2}\nabla^2\Phi_k(x^{k,t})^{-1}\nabla\Phi_k(x^{k,t})\\
   =\ &\textstyle \nabla^2\Phi_k(x^k)^{1/2}\nabla^2\Phi_k(x^{k,t})^{-1}\nabla^2\Phi_k(x^k)^{1/2}\nabla^2\Phi_k(x^k)^{-1/2}\nabla\Phi_k(x^{k,t}),
\end{align*}
we have
\begin{align*}
   \textstyle \anorm{d^{k,t}} \le \ &\textstyle \|\nabla^2\Phi_k(x^k)^{1/2}\nabla^2\Phi_k(x^{k,t})^{-1}\nabla^2\Phi_k(x^k)^{1/2}\|_{\op}\|\nabla^2\Phi_k(x^k)^{-1/2}\nabla\Phi_k(x^{k,t})\| \\
   \le \ &\textstyle \frac{4}{3}\|\nabla^2\Phi_k(x^k)^{-1/2}\nabla\Phi_k(x^{k,t})\|,
\end{align*}
where the second inequality follows from \eqref{eq:variation-norm-bound}. Consequently, it suffices to bound the term $\|\nabla^2\Phi_k(x^k)^{-1/2}\nabla\Phi_k(x^{k,t})\|$. By $\eta_{k,t-1}=1$ and the Newton equation 
\begin{align*}
    \textstyle \nabla \Phi_k(x^{k,t-1}) = -\nabla^2 \Phi_k(x^{k,t-1})(x^{k,t}-x^{k,t-1})
\end{align*}
at $x^{k,t-1}$, it follows that
\begin{align*}
    \textstyle \nabla\Phi_k(x^{k,t}) = \ &\textstyle \nabla\Phi_k(x^{k,t-1}) + \int_0^1 \nabla^2 \Phi_k(x^{k,t-1}+\tau(x^{k,t}-x^{k,t-1}))(x^{k,t}-x^{k,t-1}) \mathrm{d}\tau \\
    = \ &\textstyle -\nabla^2 \Phi_k(x^{k,t-1})(x^{k,t}-x^{k,t-1})+\int_0^1 \nabla^2 \Phi_k(x^{k,t-1}+\tau(x^{k,t}-x^{k,t-1}))(x^{k,t}-x^{k,t-1}) \mathrm{d}\tau \\
    = \ &\textstyle (\int_0^1\nabla^2\Phi_k(x^{k,t-1}+\tau(x^{k,t}-x^{k,t-1}))\mathrm{d}\tau- \nabla^2\Phi_k(x^{k,t-1}))\cdot(x^{k,t}-x^{k,t-1})
\end{align*}
Therefore, we obtain
\begin{align*}
   &\textstyle \|\nabla^2\Phi_k(x^k)^{-1/2}\nabla\Phi_k(x^{k,t})\| \\
   = \ &\textstyle \|\nabla^2\Phi_k(x^k)^{-1/2}(\int_0^1\nabla^2\Phi_k(x^{k,t-1}+\tau(x^{k,t}-x^{k,t-1})) \mathrm{d}\tau-\nabla^2\Phi_k(x^{k,t-1}))(x^{k,t}-x^{k,t-1})\|.
\end{align*}
which further implies
\begin{align*}
    &\textstyle \|\nabla^2\Phi_k(x^k)^{-1/2}\nabla\Phi_k(x^{k,t})\| \\
    \le \ &\textstyle \|\nabla^2\Phi_k(x^k)^{-1/2}(\int_0^1\nabla^2\Phi_k(x^{k,t-1}+\tau(x^{k,t}-x^{k,t-1})) \mathrm{d}\tau-\nabla^2\Phi_k(x^{k,t-1})) \nabla^2\Phi_k(x^k)^{-1/2}\|_{\op}\\
    &\cdot\anorm{x^{k,t}-x^{k,t-1}},
\end{align*}
since $\nabla^2\Phi_k(x^k)^{-1/2} \nabla^2 \Phi_k(x^k)^{1/2}(x^{k,t}-x^{k,t-1})=x^{k,t}-x^{k,t-1}$ and $\|\nabla^2 \Phi_k(x^k)^{1/2}(x^{k,t}-x^{k,t-1})\| = \anorm{x^{k,t}-x^{k,t-1}}$. In the following, we bound the first term on the right-hand side of the above inequality. Note that
\begin{align*}
    &\textstyle \int_0^1\nabla^2\Phi_k(x^{k,t-1}+\tau(x^{k,t}-x^{k,t-1})) \mathrm{d}\tau-\nabla^2\Phi_k(x^{k,t-1})\\
    = \ &\textstyle \tfrac{1}{2}(\nabla^2\Phi_k(x^{k,t})-\nabla^2\Phi_k(x^{k,t-1})) \\
    &\textstyle + \int_0^1\nabla^2\Phi_k(x^{k,t-1}+\tau(x^{k,t}-x^{k,t-1}))-(1-\tau)\nabla^2\Phi_k(x^{k,t-1})-\tau\nabla^2\Phi_k(x^{k,t}) \mathrm{d}\tau.
\end{align*}
We thus have
\begin{align*}
&\textstyle \|\nabla^2\Phi_k(x^k)^{-1/2}(\int_0^1\nabla^2\Phi_k
(x^{k,t-1}+\tau(x^{k,t}-x^{k,t-1})) \mathrm{d}\tau-\nabla^2\Phi_k(x^{k,t-1})) \nabla^2\Phi_k(x^k)^{-1/2}\|_{\op}\\
\le \ &\textstyle \tfrac{1}{2}\|\nabla^2\Phi_k(x^k)^{-1/2}(\nabla^2\Phi_k(x^{k,t})-\nabla^2\Phi_k(x^{k,t-1}))\nabla^2\Phi_k(x^k)^{-1/2}\|_{\op}\\
&\textstyle + \tfrac1\mu\int_0^1\|\nabla^2\Phi_k(x^{k,t-1}+\tau(x^{k,t}-x^{k,t-1}))-(1-\tau)\nabla^2\Phi_k(x^{k,t-1})-\tau\nabla^2\Phi_k(x^{k,t})\|_{\op} \mathrm{d}\tau\\
\le \ &\textstyle \tfrac{1}{2}\|\nabla^2\Phi_k(x^k)^{-1/2}(\nabla^2\Phi_k(x^{k,t})-\nabla^2\Phi_k(x^k)) \nabla^2\Phi_k(x^k)^{-1/2}\|_{\op}\\
&\textstyle+ \tfrac{1}{2}\|\nabla^2\Phi_k(x^k)^{-1/2}(\nabla^2\Phi_k(x^{k,t-1})-\nabla^2\Phi_k(x^k))\nabla^2\Phi_k(x^k)^{-1/2}\|_{\op}\\
&\textstyle +\tfrac{L_T}{2\mu}\|x^{k,t}-x^{k,t-1}\|^2 \int_0^1\tau(1-\tau) \mathrm{d}\tau,
\end{align*}
where the first inequality uses $\|\nabla^2\Phi_k(x^k)^{-1/2}\|_{\op}^2 \le 1/\mu$, and the second inequality uses
\begin{align*}
    \textstyle \nabla^2\Phi_k(x^{k,t})-\nabla^2\Phi_k(x^{k,t-1})=\nabla^2\Phi_k(x^{k,t})-\nabla^2\Phi_k(x^{k}) + \nabla^2\Phi_k(x^{k})-\nabla^2\Phi_k(x^{k,t-1})
\end{align*}
and \eqref{eq:chord}. Recall that \eqref{eq:gate} does not hold at either $x^{k,t-1}$ or $x^{k,t}$. By \eqref{eq:variation-norm-bound}, it follows that
\begin{align*}
    &\textstyle \|\nabla^2\Phi_k(x^k)^{-1/2}(\int_0^1\nabla^2\Phi_k
(x^{k,t-1}+\tau(x^{k,t}-x^{k,t-1})) \mathrm{d}\tau-\nabla^2\Phi_k(x^{k,t-1})) \nabla^2\Phi_k(x^k)^{-1/2}\|_{\op} \\
\le \ &\textstyle \frac{1}{4}+\frac{L_T}{12\mu}\|x^{k,t}-x^{k,t-1}\|^2.
\end{align*}
Since $\mu\|x^{k,t}-x^{k,t-1}\|^2\le\anorm{x^{k,t}-x^{k,t-1}}^2$ and the induction hypothesis gives $\anorm{x^{k,t}-x^{k,t-1}}\le2^{1-t}\Delta \le \Delta$, we have
\begin{align*}
    \textstyle \|\nabla^2\Phi_k(x^k)^{-1/2}(\int_0^1\nabla^2\Phi_k(x^{k,t-1}+\tau(x^{k,t}-x^{k,t-1})) \mathrm{d}\tau-\nabla^2\Phi_k(x^{k,t-1})) \nabla^2\Phi_k(x^k)^{-1/2}\|_{\op} \le \frac{1}{3}.
\end{align*}
Thus, we obtain
\begin{align*}
   \textstyle \|\nabla^2\Phi_k(x^k)^{-1/2}\nabla\Phi_k(x^{k,t})\| \le \frac{1}{3}\|x^{k,t} - x^{k,t-1}\|_{k,\mu},
\end{align*}
which yields
\begin{align*}
    \textstyle \anorm{d^{k,t}} \le \frac{1}{2}\|x^{k,t} - x^{k,t-1}\|_{k,\mu} \le \tfrac{1}{2}\cdot 2^{1-t}\Delta=2^{-t}\Delta < \Delta.
\end{align*}
Therefore, $\eta_{k,t} = 1$ and hence $\anorm{x^{k,t+1}-x^{k,t}}=\anorm{d^{k,t}} \le 2^{-t}\Delta$. The induction is complete. The first inequality in \eqref{eq:contraction} now gives
\begin{align*}
   \textstyle \anorm{x^{k,t}-x^k} =  \anorm{x^{k,t}-x^{k,0}} \le\sum_{j=1}^t\anorm{x^{k,j}-x^{k,j-1}} \le\Delta\sum_{j=1}^t2^{1-j} <2\Delta. 
\end{align*}
Finally, $\mu\|x^{k,t}-x^k\|^2\le\anorm{x^{k,t}-x^k}^2$ yields $\|x^{k,t}-x^k\|^2\le\frac{4\Delta^2}{\mu}$. The proof is complete.
\end{proof}

\begin{lemma}[Small-step case: convergence]\label{lem:small-newton}
Under \Cref{ass:hessian,ass:third}, suppose $256L_T\Delta^2\le\mu^2$ and the first Newton step satisfies $\anorm{d^{k,0}}<\Delta$. Then every inner iterate $x^{k,t}$ with $t \ge 1$ at which \eqref{eq:gate} does not hold satisfies
\begin{equation*}
    \textstyle \|\nabla f(x^{k,t})\|\le\frac{L_H\Delta^2}{2\mu}4^{1-t}+4\sqrt\mu\Delta
    \ \text{and} \ \nabla^2f(x^{k,t})\succ-\mu I.
\end{equation*}
\end{lemma}
\begin{proof}
First, note that $\|\nabla f(x^{k,t})\|\le\|\nabla\Phi_k(x^{k,t})\|+2\mu\|x^{k,t}-x^k\|$. Recall that $\nabla^2 \Phi_k(x^k) \succ \mu I$. By \Cref{lem:contraction}, we have
\begin{align*}
    \textstyle \sqrt\mu\|x^{k,t}-x^{k,t-1}\|
    &\textstyle\le\anorm{x^{k,t}-x^{k,t-1}}\le2^{1-t}\Delta,\\
    \textstyle \sqrt\mu\|x^{k,t}-x^k\|
    &\textstyle\le\anorm{x^{k,t}-x^k}\le2\Delta.
\end{align*}
Note that $x^{k,t} = x^{k,t-1} - \nabla^2 \Phi_k(x^{k,t-1})^{-1}\nabla\Phi_k(x^{k,t-1})$. Consequently, applying \eqref{eq:gradient-taylor} gives
\begin{align*}
    \textstyle \|\nabla\Phi_k(x^{k,t})\|
    = \ &\textstyle\|\nabla\Phi_k(x^{k,t})-\nabla\Phi_k(x^{k,t-1})
    -\nabla^2\Phi_k(x^{k,t-1})(x^{k,t}-x^{k,t-1})\|\\
    \le \ &\textstyle\frac{L_H}{2}\|x^{k,t}-x^{k,t-1}\|^2 \\
    \le \ &\textstyle\frac{L_H\Delta^2}{2\mu}4^{1-t},
\end{align*}
which implies that $\|\nabla f(x^{k,t})\|\le\frac{L_H\Delta^2}{2\mu}4^{1-t}+4\sqrt\mu\Delta$. Finally, since \eqref{eq:gate} does not hold at $x^{k,t}$, then
\begin{align*}
    \textstyle \nabla^2f(x^{k,t}) \textstyle\succ\frac34\nabla^2f(x^k)-\frac{\mu}{4}I \succ-\mu I,
\end{align*}
which completes the proof.
\end{proof}

\begin{theorem}\label{thm:acc}
Under \Cref{ass:hessian,ass:third}, choose $\mu=\min\{(L_T\epsg^2)^{1/3},\epsH/2\}$ and $\Delta=\frac{\mu}{16\sqrt{L_T}}$. Then \Cref{alg:accppnewton} returns an $(\epsg,\epsH)$-approximate second-order stationary point using at most
\begin{align*}
    \widetilde\cO\big((f(x^0)-f_{\inf})\max\{L_T^{1/3}\epsg^{-4/3},L_T\epsH^{-2}\}\big)
\end{align*}
second-order oracle evaluations.
\end{theorem}
\begin{proof}
With the choices of $\mu$ and $\Delta$, the negative-curvature direction update yields a decrease of at least $2\Delta^2/5$ by \Cref{lem:decrease-initial} and \Cref{lem:decrease}. \Cref{lem:initial-newton} gives a decrease of at least $2\Delta^2/5$ in the large-step case. In the small-step case, when \eqref{eq:gate} does not hold, \Cref{lem:small-newton} and the choices of $\mu$ and $\Delta$ give
\begin{align*}
    \textstyle \|\nabla f(x^{k,t})\| \le \frac{L_H\mu}{512L_T}4^{1-t}+\frac{\epsg}{4} \le \frac{\epsg}{512}+\frac{\epsg}{4} < \epsg,
\end{align*}
where the second inequality holds whenever $t\ge 1+\big\lceil\log_4(1+\tfrac{L_H\mu}{L_T\epsg})\big\rceil$. Consequently, in the small-step case, the inner iterates either enter the curvature branch or yield an $(\epsg,\epsH)$-approximate second-order stationary point within $1+\big\lceil\log_4(1+\tfrac{L_H\mu}{L_T\epsg})\big\rceil$ Newton steps. After $K$ outer updates, we have
\begin{align*}
    \textstyle \frac{2K\Delta^2}{5} \le \sum_{k=0}^{K-1}(f(x^k)-f(x^{k+1})) = f(x^0)-f(x^K) \le f(x^0)-f_{\inf}.
\end{align*}
Therefore, there are at most $\lfloor 5(f(x^0)-f_{\inf})/(2\Delta^2)\rfloor$ outer iterations, and each inner loop uses at most $1+\big\lceil\log_4(1+\frac{L_H\mu}{L_T\epsg})\big\rceil$ second-order oracle evaluations. Finally, substituting $\Delta^2=\mu^2/(256L_T)$ and $L_T/\mu^2=\max\{L_T^{1/3}\epsg^{-4/3},4L_T\epsH^{-2}\}$ completes the proof.
\end{proof}

\section{Accelerated Proximal Point Newton--CG Method}\label{sec:binary-cg}
The previous sections establish the proximal point Newton framework and its accelerated convergence guarantees. However, the accelerated method relies on exact eigenvalue computations and Newton system solves, as well as prior knowledge of smoothness parameters. We now address these computational limitations by replacing the exact matrix operations with Newton–CG and randomized Lanczos procedures, and by selecting the regularization parameters adaptively. Specifically, we implement \Cref{alg:accppnewton} using capped CG and the Lanczos method, following the Newton--CG method of \cite{royeroneillwright2018}, and term our method the accelerated  Proximal Point Newton--CG Method. The proximal subproblem and relative negative-curvature detection remain those of \Cref{sec:accelerated}. We use capped CG to replace the exact Newton solve and a minimum eigenvalue oracle based on the Lanczos method to replace the eigenvalue computation. The two procedures and their guarantees are given in \Cref{app:capped-cg,app:meo}. The complete algorithm is presented in \Cref{alg:accppncg}.

\begin{algorithm}[ht]
\caption{Accelerated Proximal Point Newton--CG Method}\label{alg:accppncg}
\begin{algorithmic}[1]
\Require accuracies $\epsg>0$ and $\epsH>0$, Lipschitz constant estimates $L_{\rm in}\ge L_{\min}>0$.
\State \textbf{Initialization:} initial point $x^0\in\bR^d$, outer iteration count $k=0$, $L=L_{\rm in}$.\label{line:ncg-initialize}
\While{no stationary point has been returned}
  \State Initialize the parameters $\mu,\Delta,\Phi_k,x^{k,0}$ according to \eqref{eq:binary-radius-choice} for the inner loop.\label{line:ncg-parameters}
  \For{$t=0,1,2,\ldots$}\label{line:ncg-inner}
    \State Apply capped CG (\Cref{procedure:capped-cg}) with $\mu$ and \eqref{eq:ncg-cg-call}.\label{line:ncg-cg}
    \If{\texttt{NC} is returned}\label{line:ncg-cg-nc}
      \State Normalize the output to $u^{k,t}$, set $s^{k,t}=1$, and go to \textbf{Curvature update}.
    \EndIf
    \State Set $d^{k,t}$ to the \texttt{SOL} output.\label{line:ncg-sol}
    \State Apply minimum eigenvalue oracle (\Cref{procedure:meo}) with accuracy $\mu/16$ to \eqref{eq:ncg-meo-call}.\label{line:ncg-meo}
    \If{\texttt{NC} is returned}\label{line:ncg-meo-nc}
      \State Retain $u^{k,t},s^{k,t}$ and go to \textbf{Curvature update}.
    \EndIf
    \If{$t=1$ and $\anorm{x^{k,1}-x^k}=\Delta$}\label{line:ncg-large}
      \State Update $\tilde{x}=x^{k,1}$ and \textbf{break}.
    \EndIf
    \If{$t\ge1$ and \eqref{eq:binary-contract} fails}\label{line:ncg-contract}
      \State Reject the inner loop and \textbf{break}.
    \EndIf
    \State Set the stepsize $\eta_{k,t}=\min\{1,\Delta/\anorm{d^{k,t}}\}$, update $x^{k,t+1}=x^{k,t}-\eta_{k,t}d^{k,t}$, and \textbf{continue}.\label{line:ncg-step}
    \State \textbf{Curvature update:} update $\tilde{x}$ by \eqref{eq:binary-candidates} and \textbf{break}.\label{line:ncg-curvature}
  \EndFor
  \If{$\tilde{x}$ is updated in the inner loop and \eqref{eq:binary-curvature} holds}\label{line:ncg-acceptance}
    \State Set $x^{k+1} = \tilde{x}$, $k=k+1$ and $L=\max\{L_{\min},L/2\}$.\label{line:ncg-accept}
  \Else
    \State Reject the inner loop; retain $x^k$ and set $L=2L$.\label{line:ncg-reject}
  \EndIf
\EndWhile
\end{algorithmic}
\end{algorithm}

Lines~\ref{line:ncg-initialize}--\ref{line:ncg-inner} initialize the adaptive parameters and start an inner loop at the current outer iterate $x^k$. With the current Lipschitz constant estimate $L\ge L_{\min}>0$, initialized as $L=L_{\rm in}$, we set the regularization parameter, radius, and associated proximal function as follows
\begin{equation}\label{eq:binary-radius-choice}
    \textstyle\mu=\min\{(L\epsg^2)^{1/3},\epsH/2\},\ \Delta=\frac{\mu}{16\sqrt L},\ \Phi_k(x)=f(x)+\mu\|x-x^k\|^2,\ x^{k,0}=x^k.
\end{equation}
These quantities remain fixed throughout the inner loop. At $x^{k,t}$, line~\ref{line:ncg-cg} applies capped CG with regularization parameter $\mu$ and
\begin{equation}\label{eq:ncg-cg-call}
    \textstyle H=\nabla^2f(x^{k,t}),\ g=-\nabla\Phi_k(x^{k,t}),\ \zeta=\frac{1}{32}.
\end{equation}
If it returns \texttt{NC}, we normalize its output to $u^{k,t}$, set $s^{k,t}=1$, and proceed to the curvature update. In this case, the normalized \texttt{NC} output of capped CG satisfies
\begin{equation}\label{eq:binary-nc}
    \textstyle\|u^{k,t}\|=1,\ \inner{\nabla^2f(x^{k,t})u^{k,t}}{u^{k,t}}<-\mu.
\end{equation}
Otherwise, we retain the \texttt{SOL} output as the inexact Newton direction $d^{k,t}$, which satisfies
\begin{equation}\label{eq:binary-sol}
\begin{aligned}
    \textstyle\|\nabla^2\Phi_k(x^{k,t})d^{k,t}-\nabla\Phi_k(x^{k,t})\|&\textstyle\le\frac{1}{64}\min\{\mu\|d^{k,t}\|,\|\nabla\Phi_k(x^{k,t})\|\},\\
    \textstyle\inner{\nabla^2\Phi_k(x^{k,t})d^{k,t}}{d^{k,t}}&\textstyle\ge\mu\|d^{k,t}\|^2
\end{aligned}
\end{equation}
by \Cref{lem:ccg-interface} and \eqref{eq:ncg-cg-call}. After a $\texttt{SOL}$ output, we first apply the minimum eigenvalue oracle with accuracy $\mu/16$ to the two matrices
\begin{equation}\label{eq:ncg-meo-call}
    \textstyle H=\nabla^2f(x^k)+\frac{29}{32}\mu I+4s\big(\nabla^2f(x^{k,t})-\nabla^2f(x^k)\big),\ s\in\{-1,1\}
\end{equation}
to detect a relative negative-curvature direction. If the minimum eigenvalue oracle returns \texttt{NC} for some $s\in\{-1,1\}$, we set $s^{k,t}=s$, let $u^{k,t}$ be the returned unit negative-curvature vector, and proceed to the curvature update, as prescribed by line~\ref{line:ncg-meo-nc}. In this case, we have
\begin{equation}\label{eq:ncg-gate}
    \textstyle\inner{\big(\nabla^2f(x^k)+\mu I+4s^{k,t}(\nabla^2f(x^{k,t})-\nabla^2f(x^k))\big)u^{k,t}}{u^{k,t}}\le\mu/16.
\end{equation}
We remark that a capped-CG \texttt{NC} output at $t\ge1$ also satisfies \eqref{eq:ncg-gate} with $s^{k,t}=1$, since $\nabla^2f(x^k)\succ-\mu I$. If no negative curvature is returned, i.e., when both signs return \texttt{CERT}, then each matrix in \eqref{eq:ncg-meo-call} has a smallest eigenvalue of at least $-\mu/16$, which implies
\begin{equation}\label{eq:ncg-certificate}
    \textstyle\lambda_{\min}\big(\nabla^2f(x^k)+\mu I+4s(\nabla^2f(x^{k,t})-\nabla^2f(x^k))\big)\ge\mu/32,\ s\in{-1,1}.
\end{equation}
Consequently, \eqref{eq:relative-reserve} and \eqref{eq:variation-norm-bound} hold, and $\nabla^2f(x^{k,t})\succ-\mu I$. At $t=1$, we then check whether the first Newton step reaches the fixed radius and, if so, introduce an auxiliary point $\tilde{x}=x^{k,1}$. Otherwise, line~\ref{line:ncg-contract} allows the small-step inner loop to continue only when
\begin{equation}\label{eq:binary-contract}
    \textstyle\anorm{d^{k,t}}\le\frac{1}{2}\anorm{x^{k,t}-x^{k,t-1}},\ \forall t\ge1.
\end{equation}
If this condition fails, the inner loop is rejected. At $t=0$, or when the condition holds, we apply the inexact Newton direction, update $x^{k,t+1}$, and continue the inner loop. Whenever either procedure returns \texttt{NC}, define the auxiliary point
\begin{equation}\label{eq:binary-candidates}
 \textstyle \tilde{x}\in
 \begin{cases}
 \argmin_{x\in\{x^k\pm\sqrt{\mu/L}\,u^{k,0}\}}f(x),&t=0,\\
 \argmin_{x\in\{x^k-4s^{k,t}(x^{k,t}-x^k),\,x^k+4s^{k,t}(x^{k,t}-x^k)\pm\sqrt{\mu/L}\,u^{k,t}\}}f(x),&t\ge1.
 \end{cases}
\end{equation}
Compared with the previous two algorithms, an auxiliary point $\tilde{x}$ is introduced. We update the outer iterate only when an auxiliary point $\tilde{x}$ has been produced and the following sufficient decrease in the function value is obtained:
\begin{equation}\label{eq:binary-curvature}
    \textstyle f(\tilde{x}) - f(x^k)\le-\frac{\Delta^2}{4}.
\end{equation}
In this case, we update the outer iterate $x^{k+1}=\tilde{x}$ and set $L=\max\{L_{\min},L/2\}$. Retain the iterate $x^k$ and double $L$. The next inner loop then restarts from $x^k$. As we will show later, when the Lipschitz constant estimate satisfies $L\ge L_T$, rejection cannot occur, and every inner loop terminates within a logarithmic number of iterations. We now provide the complexity analysis of \Cref{alg:accppncg}.

\begin{lemma}\label{lem:binary-initial-curvature}
Under \Cref{ass:third}, suppose $L\ge L_T$. If either procedure returns \texttt{NC} at $x^{k,0}=x^k$, then the auxiliary point $\tilde{x}$ in \eqref{eq:binary-candidates} satisfies
\begin{equation*}
    f(\tilde{x})-f(x^k)\le-\frac{67\mu^2}{160L}.
\end{equation*}
\end{lemma}
\begin{proof}
By the definition of $\tilde{x}$ in \eqref{eq:binary-candidates}, we have
\begin{align*}
    \textstyle f(\tilde{x})-f(x^k)\le\ &\textstyle\frac12\big(f(x^k+\sqrt{\mu/L}\,u^{k,0})+f(x^k-\sqrt{\mu/L}\,u^{k,0})\big)-f(x^k)\\
    \le\ &\textstyle\frac{\mu}{2L}\inner{\nabla^2f(x^k)u^{k,0}}{u^{k,0}}+\frac{\mu^2}{20L},
\end{align*}
where the last inequality follows from \eqref{eq:tensor-function-taylor}, $\|u^{k,0}\|=1$, and $L_T\le L$. By either \eqref{eq:binary-nc} or \eqref{eq:ncg-gate}, we have $\inner{\nabla^2f(x^k)u^{k,0}}{u^{k,0}}\le-15\mu/16$. Consequently, it follows that
\begin{align*}
    \textstyle f(\tilde{x})-f(x^k)\le-\frac{15\mu^2}{32L}+\frac{\mu^2}{20L}=-\frac{67\mu^2}{160L},
\end{align*}
which completes the proof.
\end{proof}

\begin{lemma}\label{lem:binary-relative-curvature}
Under \Cref{ass:third}, suppose $L\ge L_T$ and $\nabla^2\Phi_k(x^k)\succ\mu I$. If either procedure returns \texttt{NC} at an inner iterate $x^{k,t}$ with $t\ge1$ and $\anorm{x^{k,t}-x^k}\le2\Delta$, then the auxiliary point $\tilde{x}$ in \eqref{eq:binary-candidates} satisfies
\begin{equation*}
    f(\tilde{x})-f(x^k)\le-\frac{27\mu^2}{320L}.
\end{equation*}
\end{lemma}
\begin{proof}
By the definition of $\tilde{x}$ in \eqref{eq:binary-candidates}, we have
\begin{align*}
    \textstyle f(\tilde{x})-f(x^k)\le \ &\textstyle\frac12f(x^k-4s^{k,t}(x^{k,t}-x^k))+\frac14f(x^k+4s^{k,t}(x^{k,t}-x^k)+\sqrt{\mu/L}\,u^{k,t})\\
    &\textstyle+\frac14f(x^k+4s^{k,t}(x^{k,t}-x^k)-\sqrt{\mu/L}\,u^{k,t})-f(x^k).
\end{align*}
By applying \eqref{eq:tensor-function-taylor} and bounding the quartic terms as in the proof of \Cref{lem:decrease}, it follows that
\begin{align*}
    \textstyle f(\tilde{x})-f(x^k)\le\ &\textstyle8\anorm{x^{k,t}-x^k}^2-\frac{68\mu}{5}\|x^{k,t}-x^k\|^2+\frac{64L}{5}\|x^{k,t}-x^k\|^4+\frac{\mu^2}{40L}\\
    &\textstyle+\frac{\mu}{4L}\big(\inner{\nabla^2f(x^k)u^{k,t}}{u^{k,t}}+4s^{k,t}\nabla^3f(x^k)[x^{k,t}-x^k,u^{k,t},u^{k,t}]\big).
\end{align*}
Since either procedure returns \texttt{NC}, \eqref{eq:ncg-gate} holds. By \eqref{eq:taylor} and $\|u^{k,t}\|=1$, we have
\begin{align*}
    \textstyle\inner{\nabla^2f(x^k)u^{k,t}}{u^{k,t}}+4s^{k,t}\nabla^3f(x^k)[x^{k,t}-x^k,u^{k,t},u^{k,t}]\le-\frac{15\mu}{16}+2L_T\|x^{k,t}-x^k\|^2.
\end{align*}
Combining the above inequalities and using $L_T\le L$ gives
\begin{align*}
    \textstyle f(\tilde{x})-f(x^k)\le8\anorm{x^{k,t}-x^k}^2-\frac{131\mu}{10}\|x^{k,t}-x^k\|^2+\frac{64L}{5}\|x^{k,t}-x^k\|^4-\frac{67\mu^2}{320L}.
\end{align*}
By \eqref{eq:binary-radius-choice}, we have $256L\Delta^2=\mu^2$. Since $\nabla^2\Phi_k(x^k)\succ\mu I$ and $\anorm{x^{k,t}-x^k}\le2\Delta$, it follows that
\begin{align*}
    \textstyle\frac{64L}{5}\|x^{k,t}-x^k\|^4\le\frac{256L\Delta^2}{5\mu}\|x^{k,t}-x^k\|^2=\frac\mu5\|x^{k,t}-x^k\|^2.
\end{align*}
Therefore, we obtain
\begin{align*}
    \textstyle f(\tilde{x})-f(x^k)\le32\Delta^2-\frac{67\mu^2}{320L}=-\frac{27\mu^2}{320L},
\end{align*}
which completes the proof.
\end{proof}

\begin{lemma}[Large-step case]\label{lem:binary-initial-newton}
Under \Cref{ass:hessian,ass:third}, suppose \eqref{eq:ncg-certificate} holds whenever both signs return \texttt{CERT} in the inner loop and the first Newton step satisfies $\anorm{d^{k,0}}\ge\Delta$. Then the inner loop ends no later than $t=1$. If $L\ge L_T$, the inner loop is not rejected, and any auxiliary point $\tilde{x}$ produced satisfies
\begin{equation*}
    f(\tilde{x})-f(x^k)\le-\frac{\Delta^2}{4}.
\end{equation*}
\end{lemma}
\begin{proof}
By \Cref{alg:accppncg}, the inner loop ends no later than $t=1$. It remains to prove the sufficient decrease when $L\ge L_T$. If either procedure returns \texttt{NC} at $x^{k,1}$, then $\anorm{x^{k,1}-x^k}=\Delta\le2\Delta$, and \Cref{lem:binary-relative-curvature} gives
\begin{align*}
    \textstyle f(\tilde{x})-f(x^k)\le-\frac{27\mu^2}{320L}\le-\frac{\Delta^2}{4},
\end{align*}
where the last inequality follows from $256L\Delta^2=\mu^2$. Otherwise, both signs return \texttt{CERT} and $\tilde{x}=x^{k,1}$, unless a stationary point is returned. By the definition of the proximal function and \eqref{eq:chord}, we have
\begin{align*}
    &\textstyle f(x^{k,1})-f(x^k)\\
    \le \ &\textstyle\Phi_k(x^{k,1})-\Phi_k(x^k)\\
    \le\ &\textstyle\inner{\nabla\Phi_k(x^k)}{x^{k,1}-x^k}+\frac13\anorm{x^{k,1}-x^k}^2+\frac16\inner{\nabla^2\Phi_k(x^{k,1})(x^{k,1}-x^k)}{x^{k,1}-x^k}+\frac{L_T}{24}\|x^{k,1}-x^k\|^4.
\end{align*}
We first bound the inner product term $\inner{\nabla\Phi_k(x^k)}{x^{k,1}-x^k}$. For every computed Newton direction, \eqref{eq:binary-sol} gives
\begin{equation*}
    \textstyle\|\nabla^2\Phi_k(x^k)^{-1/2}(\nabla^2\Phi_k(x^{k,t})d^{k,t}-\nabla\Phi_k(x^{k,t}))\| \le \frac1{\sqrt\mu}\|\nabla^2\Phi_k(x^{k,t})d^{k,t}-\nabla\Phi_k(x^{k,t})\| \le \textstyle \frac{\sqrt\mu}{64}\|d^{k,t}\|.
\end{equation*}
We thus have
\begin{equation}\label{eq:binary-dual}
\begin{aligned}
    \textstyle\|\nabla^2\Phi_k(x^k)^{-1/2}(\nabla^2\Phi_k(x^{k,t})d^{k,t}-\nabla\Phi_k(x^{k,t}))\| \le \frac1{64}\anorm{d^{k,t}}
\end{aligned}
\end{equation}
since $\nabla^2\Phi_k(x^k)\succ\mu I$,
By $x^{k,1}=x^k-\eta_{k,0}d^{k,0}$ and \eqref{eq:binary-dual} at $t=0$, we obtain
\begin{align*}
    \textstyle\inner{\nabla\Phi_k(x^k)}{x^{k,1}-x^k}=\ &\textstyle-\frac1{\eta_{k,0}}\anorm{x^{k,1}-x^k}^2-\inner{\nabla^2\Phi_k(x^k)d^{k,0}-\nabla\Phi_k(x^k)}{x^{k,1}-x^k}\\
    \le\ &\textstyle-\frac{63}{64\eta_{k,0}}\anorm{x^{k,1}-x^k}^2 \\
    \le \ &\textstyle -\frac{63}{64}\anorm{x^{k,1}-x^k}^2,
\end{align*}
where the last inequality uses $\eta_{k,0}\le1$. Moreover, \eqref{eq:relative-reserve} gives
\begin{align*}
    \textstyle\frac16\inner{\nabla^2\Phi_k(x^{k,1})(x^{k,1}-x^k)}{x^{k,1}-x^k}\le\frac5{24}\anorm{x^{k,1}-x^k}^2.
\end{align*}
Finally, since $\anorm{x^{k,1}-x^k}=\Delta$ and $\nabla^2\Phi_k(x^k)\succ\mu I$, the choice \eqref{eq:binary-radius-choice} gives
\begin{align*}
    \textstyle\frac{L_T}{24}\|x^{k,1}-x^k\|^4\le\frac{L_T\Delta^2}{24\mu^2}\anorm{x^{k,1}-x^k}^2\le\frac1{6144}\anorm{x^{k,1}-x^k}^2.
\end{align*}
Combining the above inequalities yields
\begin{align*}
    \textstyle f(x^{k,1})-f(x^k)\le-\frac{\Delta^2}{4}.
\end{align*}
Thus, \eqref{eq:binary-curvature} holds in either case, which completes the proof.
\end{proof}

In the following, we focus on the small-step case. We first show that every subsequent inner step takes a full Newton update and the inner iterates remain within the fixed radius required in \Cref{lem:binary-relative-curvature}.
\begin{lemma}[Small-step case: contraction]\label{lem:binary-refinement}
Under \Cref{ass:hessian,ass:third}, suppose \eqref{eq:ncg-certificate} holds whenever both signs return \texttt{CERT} in the inner loop and the first Newton step satisfies $\anorm{d^{k,0}}<\Delta$. Then every reached inner iterate $x^{k,t}$ with $t\ge1$ satisfies
\begin{equation}\label{eq:binary-confine}
    \anorm{x^{k,t}-x^{k,t-1}}\le2^{1-t}\Delta,\ \anorm{x^{k,t}-x^k}\le2\Delta,\ \|x^{k,t}-x^k\|^2\le\frac{4\Delta^2}{\mu}.
\end{equation}
Every Newton update taken in this inner loop is a full step. If $L\ge L_T$, then \eqref{eq:binary-contract} holds whenever both signs return \texttt{CERT} at $x^{k,t}$ with $t\ge1$.
\end{lemma}
\begin{proof}
We first prove by induction that every reached inner iterate $x^{k,t}$ with $t\ge1$ satisfies
\begin{align*}
    \textstyle\eta_{k,t-1}=1,\ \anorm{x^{k,t}-x^{k,t-1}}\le2^{1-t}\Delta.
\end{align*}
The case $t=1$ follows from $\anorm{d^{k,0}}<\Delta$. Suppose the claim holds at index $t\ge1$ and $x^{k,t+1}$ is reached. Then \eqref{eq:binary-contract} holds, and hence
\begin{align*}
    \textstyle\anorm{d^{k,t}}\le\frac12\anorm{x^{k,t}-x^{k,t-1}}\le2^{-t}\Delta<\Delta.
\end{align*}
Therefore, we have $\eta_{k,t}=1$ and $\anorm{x^{k,t+1}-x^{k,t}}\le2^{-t}\Delta$. The induction is complete. Summing the first inequality gives
\begin{align*}
    \textstyle\anorm{x^{k,t}-x^k}\le\sum_{j=1}^t\anorm{x^{k,j}-x^{k,j-1}}\le\Delta\sum_{j=1}^t2^{1-j}<2\Delta.
\end{align*}
Since $\nabla^2\Phi_k(x^k)\succ\mu I$, it follows that $\|x^{k,t}-x^k\|^2\le4\Delta^2/\mu$. These bounds hold at every reached inner iterate, including an iterate at which the inner loop is rejected.

It remains to prove \eqref{eq:binary-contract} when $L\ge L_T$ and both signs return \texttt{CERT} at $x^{k,t}$. By the inexact Newton equation, \eqref{eq:variation-norm-bound}, and \eqref{eq:binary-dual}, we have
\begin{align*}
    \textstyle\anorm{d^{k,t}}\le\frac43\big(\|\nabla^2\Phi_k(x^k)^{-1/2}\nabla\Phi_k(x^{k,t})\|+\frac1{64}\anorm{d^{k,t}}\big).
\end{align*}
Consequently, it suffices to bound $\|\nabla^2\Phi_k(x^k)^{-1/2}\nabla\Phi_k(x^{k,t})\|$. Since $\eta_{k,t-1}=1$, we have
\begin{align*}
    \textstyle\nabla\Phi_k(x^{k,t})=\ &\textstyle\big(\int_0^1\nabla^2\Phi_k(x^{k,t-1}+\tau(x^{k,t}-x^{k,t-1}))\mathrm{d}\tau-\nabla^2\Phi_k(x^{k,t-1})\big)(x^{k,t}-x^{k,t-1})\\
    &\textstyle-\big(\nabla^2\Phi_k(x^{k,t-1})d^{k,t-1}-\nabla\Phi_k(x^{k,t-1})\big).
\end{align*}
Applying \eqref{eq:chord} to the integral and \eqref{eq:binary-dual} to the residual gives
\begin{align*}
    &\textstyle\|\nabla^2\Phi_k(x^k)^{-1/2}\nabla\Phi_k(x^{k,t})\|\\
    \le\ &\textstyle\frac12\|\nabla^2\Phi_k(x^k)^{-1/2}(\nabla^2\Phi_k(x^{k,t})-\nabla^2\Phi_k(x^{k,t-1}))\nabla^2\Phi_k(x^k)^{-1/2}\|_{\op}\anorm{x^{k,t}-x^{k,t-1}}\\
    &\textstyle+\frac{L_T}{12\mu}\|x^{k,t}-x^{k,t-1}\|^2\anorm{x^{k,t}-x^{k,t-1}}+\frac1{64}\anorm{x^{k,t}-x^{k,t-1}}.
\end{align*}
Moreover, \eqref{eq:binary-confine} and \eqref{eq:binary-radius-choice} imply
\begin{align*}
    \textstyle\frac{L_T}{12\mu}\|x^{k,t}-x^{k,t-1}\|^2\le\frac{L_T\Delta^2}{12\mu^2}\le \frac{1}{64}.
\end{align*}
Combining the above inequalities gives
\begin{align*}
    \textstyle\|\nabla^2\Phi_k(x^k)^{-1/2}\nabla\Phi_k(x^{k,t})\|\le\frac{9}{32}\anorm{x^{k,t}-x^{k,t-1}}.
\end{align*}
Therefore, we obtain
\begin{align*}
    \textstyle\frac{47}{48}\anorm{d^{k,t}}\le\frac{3}{8}\anorm{x^{k,t}-x^{k,t-1}}\le\frac{47}{96}\anorm{x^{k,t}-x^{k,t-1}},
\end{align*}
which proves \eqref{eq:binary-contract}. The proof is complete.
\end{proof}

\begin{lemma}[Small-step case: convergence]\label{lem:binary-small-newton}
Under \Cref{ass:hessian,ass:third}, suppose \eqref{eq:ncg-certificate} holds whenever both signs return \texttt{CERT} in the inner loop and the first Newton step satisfies $\anorm{d^{k,0}}<\Delta$. Then every reached inner iterate $x^{k,t}$ with $t\ge1$ at which both signs return \texttt{CERT} satisfies
\begin{equation*}
    \|\nabla f(x^{k,t})\|\le\frac{L_H\Delta^2}{2\mu}4^{1-t}+\sqrt\mu\Delta\big(4+\frac{2^{1-t}}{64}\big)\ \text{and}\ \nabla^2f(x^{k,t})\succ-\mu I.
\end{equation*}
The inner loop ends by acceptance, rejection, or returning a stationary point no later than
\begin{equation}\label{eq:binary-inner-count}
    t=1+\Big\lceil\log_4\big(1+\frac{L_H\Delta^2}{\mu\epsg}\big)\Big\rceil.
\end{equation}
\end{lemma}
\begin{proof}
First, note that $\|\nabla f(x^{k,t})\|\le\|\nabla\Phi_k(x^{k,t})\|+2\mu\|x^{k,t}-x^k\|$. By \Cref{lem:binary-refinement}, the preceding step is full. Applying \eqref{eq:gradient-taylor} and \eqref{eq:binary-sol} gives
\begin{align*}
    &\textstyle\|\nabla\Phi_k(x^{k,t})\| \\
    \le\ &\textstyle\|\nabla\Phi_k(x^{k,t})-\nabla\Phi_k(x^{k,t-1})-\nabla^2\Phi_k(x^{k,t-1})(x^{k,t}-x^{k,t-1})\|+\|\nabla^2\Phi_k(x^{k,t-1})d^{k,t-1}-\nabla\Phi_k(x^{k,t-1})\|\\
    \le\ &\textstyle\frac{L_H}{2}\|x^{k,t}-x^{k,t-1}\|^2+\frac\mu{64}\|x^{k,t}-x^{k,t-1}\|.
\end{align*}
Moreover, \eqref{eq:binary-confine} gives $\sqrt\mu\|x^{k,t}-x^{k,t-1}\|\le2^{1-t}\Delta$ and $\sqrt\mu\|x^{k,t}-x^k\|\le2\Delta$, which implies that
\begin{align*}
    \textstyle\|\nabla f(x^{k,t})\|\le\frac{L_H\Delta^2}{2\mu}4^{1-t}+\sqrt\mu\Delta\big(4+\frac{2^{1-t}}{64}\big).
\end{align*}
Since both signs return \texttt{CERT}, \eqref{eq:ncg-certificate} also gives $\nabla^2f(x^{k,t})\succ-\mu I$.

It remains to bound the number of inner iterations. By the parameter choice \eqref{eq:binary-radius-choice}, we have $256L\Delta^2=\mu^2$, $16\sqrt\mu\Delta=\frac{\mu^{3/2}}{\sqrt L}\le\epsg$, and $2\mu\le\epsH$. At the index in \eqref{eq:binary-inner-count}, the first term in the gradient bound is at most $\epsg/2$. We thus obtain
\begin{align*}
    \textstyle\|\nabla f(x^{k,t})\|\le\frac{\epsg}{2}+\frac{\epsg}{16}\big(4+\frac1{64}\big)<\epsg,\ \nabla^2f(x^{k,t})\succ-\mu I\succeq-\frac{\epsH}{2}I.
\end{align*}
Thus, if the inner loop has not already ended, either a negative-curvature direction is returned or $x^{k,t}$ is an $(\epsg,\epsH)$-approximate second-order stationary point. The stopping test is applied after both signs return \texttt{CERT} and before the large-step or contraction test. If $\nabla\Phi_k(x^{k,t})=0$, the same test holds by $\|\nabla f(x^{k,t})\|\le2\mu\|x^{k,t}-x^k\|\le4\sqrt\mu\Delta$, including $t=0$ and the first large step. The proof is complete.
\end{proof}

\begin{theorem}\label{thm:binary}
Under \Cref{ass:hessian,ass:third}, fix $\delta\in(0,1)$ and use the failure probability parameter $\delta/[\ell(\ell+1)]$ for the $\ell$th call of \Cref{procedure:meo}, including calls in rejected inner loops. With probability at least $1-\delta$, \Cref{alg:accppncg} returns an $(\epsg,\epsH)$-approximate second-order stationary point, every value of $L$ satisfies
\begin{equation}\label{eq:binary-bound}
    L_{\min}\le L\le\max\{L_{\rm in},2L_T\},
\end{equation}
and the number of second-order oracle evaluations is
\begin{equation*}
    \widetilde\cO\big((f(x^0)-f_{\inf})\max\{\max\{L_{\rm in},2L_T\}^{1/3}\epsg^{-4/3},\max\{L_{\rm in},2L_T\}\epsH^{-2}\}\big).
\end{equation*}
\end{theorem}
\begin{proof}
We first show that all calls of \Cref{procedure:meo} satisfy \eqref{eq:meo-accuracy} with probability at least $1-\delta$. At call $\ell$, the starting vector is drawn independently and uniformly from the unit sphere. Conditional on all preceding calls, \Cref{lem:meo-guarantee} bounds the failure probability by $\delta/[\ell(\ell+1)]$. Since the call counter is not reset after rejection and $\sum_{\ell=1}^{\infty}\delta/(\ell(\ell+1))=\delta$, a union bound gives the claimed probability. In the following, we consider the event on which every call satisfies \eqref{eq:meo-accuracy}. In particular, \eqref{eq:ncg-certificate} holds whenever both signs return \texttt{CERT}.

To bound the number of second-order oracle evaluations, we first bound the cost of each inner loop and then the number of inner loops. Suppose $K$ outer updates have been accepted and $U$ inner loops have been rejected. Then at most $K+U+1$ inner loops have started. Each inner loop uses at most three additional function evaluations for the negative-curvature candidates. By \Cref{lem:binary-initial-newton,lem:binary-small-newton}, the total number of evaluations is at most
\begin{equation}\label{eq:binary-N}
    \textstyle1+(K+U+1)\Big[5+\Big\lceil\log_4\big(1+\frac{L_H\min\{(L_{\min}\epsg^2)^{1/3},\epsH/2\}}{256L_{\min}\epsg}\big)\Big\rceil\Big],
\end{equation}
where we use $\Delta^2/\mu=\min\{(L\epsg^2)^{1/3},\epsH/2\}/(256L)$ and $L\ge L_{\min}$. The function evaluations at the candidates are counted without computing unused derivatives. It remains to bound $K$ and $U$.

We first prove \eqref{eq:binary-bound}. If $L\ge L_T$, \Cref{lem:binary-initial-curvature,lem:binary-relative-curvature,lem:binary-initial-newton,lem:binary-refinement} imply that every auxiliary point satisfies \eqref{eq:binary-curvature} and every small-step contraction test holds. Hence, rejection cannot occur. Doubling $L$ after a rejection therefore gives $L<2L_T$, whereas accepting an outer update does not increase $L$. This proves \eqref{eq:binary-bound}. By the acceptance condition \eqref{eq:binary-curvature}, we then have
\begin{align*}
    \textstyle\frac{K}{1024\max\{\max\{L_{\rm in},2L_T\}^{1/3}\epsg^{-4/3},4\max\{L_{\rm in},2L_T\}\epsH^{-2}\}}\le\sum_{k=0}^{K-1}\big(f(x^k)-f(x^{k+1})\big)\le f(x^0)-f_{\inf},
\end{align*}
where we use $4/\Delta^2=1024\max\{L^{1/3}\epsg^{-4/3},4L\epsH^{-2}\}$. Consequently, we obtain
\begin{align*}
    \textstyle K\le1024(f(x^0)-f_{\inf})\max\{\max\{L_{\rm in},2L_T\}^{1/3}\epsg^{-4/3},4\max\{L_{\rm in},2L_T\}\epsH^{-2}\}.
\end{align*}
To bound $U$, note that rejection doubles $L$ and acceptance decreases $L$ by at most a factor of two. Thus, it follows that
\begin{align*}
    \textstyle L_{\rm in}2^{U-K}\le L\le\max\{L_{\rm in},2L_T\},\ U\le K+\log_2\frac{\max\{L_{\rm in},2L_T\}}{L_{\rm in}}.
\end{align*}
Substituting the bounds on $K$ and $U$ into \eqref{eq:binary-N} gives the stated complexity. Since every inner loop is finite, the algorithm must return a stationary point. The stopping test gives $\|\nabla f(x^{k,t})\|\le\epsg$, and \eqref{eq:ncg-certificate} and \eqref{eq:binary-radius-choice} give $\nabla^2f(x^{k,t})\succ-\mu I\succeq-(\epsH/2)I$. The proof is complete.
\end{proof}

\begin{remark}[Optional line search]\label{rem:binary-linesearch}
After the \texttt{SOL} output at $t=0$, one may try $x^k-\beta^j d^{k,0}$, $j=0,\ldots,J_{\max}$, with $0<\beta<1$ and $0<c<1/2$, before relative negative-curvature detection. Evaluate derivatives only if the candidate has function value at most $f(x^k)$. When its gradient norm is at most $\epsg$, call \Cref{procedure:meo} with accuracy $\epsH/2$ and
\begin{equation}\label{eq:ncg-linesearch-call}
    \textstyle H=\nabla^2f(x^k-\beta^j d^{k,0})+\frac{\epsH}{2}I.
\end{equation}
Return the candidate only if this call returns no negative curvature; its eigenvalue error is at most $\epsH/4$, so the shift gives $\lambda_{\min}(\nabla^2f(x^k-\beta^j d^{k,0}))>-\epsH$. If it returns \texttt{NC}, do not use its direction. A candidate not returned as stationary is accepted as an outer update only if
\begin{equation*}
    \textstyle f(x^k)-f(x^k-\beta^j d^{k,0})\ge\max\{c\beta^j\inner{\nabla f(x^k)}{d^{k,0}},\Delta^2/4\}.
\end{equation*}
An unsuccessful search resumes \Cref{alg:accppncg} without changing $x^k$ or $L$; a successful update uses its usual acceptance update. Include these calls in the same failure-probability allocation. This optional search adds at most $J_{\max}+1$ candidate evaluations per inner loop, preserves \eqref{eq:binary-curvature}, and evaluates derivatives only in the initial sublevel set. Since $2\mu\le\epsH$, its eigenvalue calls satisfy the same product bound. Thus, for fixed $J_{\max}$, both complexity rates remain unchanged.
\end{remark}

Suppose additionally that the initial sublevel set ${x\in\bR^d:f(x)\le f(x^0)}$ is bounded. We now derive the Hessian--vector product complexity. Note that \Cref{alg:accppncg} is monotone with respect to the outer iterates, and hence the outer iterates remain in the initial sublevel set. For the inner iterates, we have
\begin{align*}
    \textstyle\|x^{k,t}-x^k\|\le\frac{2\Delta}{\sqrt\mu}=\frac{\sqrt\mu}{8\sqrt L}\le\frac1{8\sqrt{2L_{\min}}}.
\end{align*}
Since $\nabla^2f(\cdot)$ is Lipschitz continuous, there exists a uniform constant $U_H>0$ such that the norm of every queried Hessian is bounded above by $U_H$. Note that the matrix in \eqref{eq:ncg-meo-call} has norm at most $9U_H+29\mu/32$, and each matrix--vector product requires at most two Hessian--vector products. By \Cref{lem:ccg-interface} and \Cref{lem:meo-guarantee}, the number of Hessian--vector products per inner iteration is
\begin{equation*}
    \textstyle \widetilde\cO\big(\min\{d,\sqrt{1+U_H/\mu}\}\big) = \widetilde\cO\big(\min\{d,\sqrt{1+U_H}\max\{\epsg^{-1/3},\epsH^{-1/2}\}\}\big)
\end{equation*}
since $\mu\ge\min\{(L_{\min}\epsg^2)^{1/3},\epsH/2\}$. Combining it with the evaluation complexity in \Cref{thm:binary} gives
\begin{equation*}
    \textstyle\widetilde\cO\big((f(x^0)-f_{\inf})\min\{d\max\{\epsg^{-4/3},\epsH^{-2}\},\sqrt{1+U_H}\max\{\epsg^{-5/3},\epsH^{-5/2}\}\}\big).
\end{equation*}
Unlike the second-order-stationarity extension of the Convex Until Proven Guilty method in \cite{carmon2017}, our method requires no prior knowledge of the true Lipschitz constants and adaptively adjusts the regularization parameters.

\section{Numerical Experiments}\label{sec:numerical}
In this section, we evaluate the numerical performance of the accelerated proximal point Newton--CG method (\texttt{APPNCG}) on the 35 unconstrained minimization problems from Mor'e, Garbow, and Hillstrom's (MGH) test set \citep{more1981testing} and a class of nonconvex regression problems. We compare our method with the adaptive regularization methods \texttt{AR3-Interp+} and \texttt{AR2-Interp} in \citet{cartishauserliuwelzelzhu2026}. For both methods, we adapt the MATLAB code released in their paper.\footnote{The authors' implementation is available at \url{https://github.com/karlwelzel/ar3-matlab}.} We also compare with the homogeneous second-order descent method (\texttt{HSODM}) from \citet{zhang2026homogeneous,he2026homogeneous} and the adaptive universal trust-region method (\texttt{UTR}) in \citet{jiang2025accelerating,jiang2026beyond}. We set $\epsg=10^{-6}$ and $\epsH=10^{-3}$. The experiments examine both computational time and the amount of oracle information. To compare methods that use different representations of the Hessian, we report the following quantity under the label \textit{Oracle Evaluations}:
\begin{equation}\label{eq:numerical-oracle-evaluations}
\textstyle N_{\mathrm{oracle}} \coloneqq N_f+N_g+N_{Hv}+dN_H,
\end{equation}
where $N_f$ and $N_g$ are the numbers of function and gradient evaluations, respectively, $N_{Hv}$ counts Hessian--vector products requested through the oracle, and $N_H$ counts explicit dense Hessian evaluations requested by the solver. All experiments are performed in MATLAB R2026a on a MacBook Pro equipped with an Apple M5 chip and 32 GB of RAM.

We begin with the 35 problems in the MGH test set, using the default dimension and starting point for each problem. The problem dimensions range from 2 to 40. We compare the methods in terms of runtime and oracle evaluations \eqref{eq:numerical-oracle-evaluations}, using the performance profiles of \citet{dolan2002benchmarking}. A performance profile shows the fraction of problems on which a method performs within a given factor of the best method for that problem. Let $\mathcal{P}$ be the set of test problems and $\mathcal{S}$ the set of methods. For each problem $p$ and method $s$, let $t_{p,s}$ denote the corresponding cost. A run is considered successful only if its output $x$ satisfies $\|\nabla f(x)\|\le \epsg$ and $\lambda_{\min}(\nabla^2 f(x))>-\epsH$. We set $t_{p,s}=+\infty$ if the run is unsuccessful.
The performance ratio is defined as $r_{p,s}\coloneqq t_{p,s}/\min_{s'\in\mathcal{S}}t_{p,s'}$ with $r_{p,s}=+\infty$ if all methods fail on problem $p$. The performance profile is then given by
\begin{equation}\label{eq:numerical-performance-profile}
    \textstyle \rho_s(\alpha)\coloneqq\frac{1}{|\mathcal{P}|}
    \big|\{p\in\mathcal{P}:r_{p,s}\le 2^\alpha\}\big|, \ \alpha\ge0.
\end{equation}
In particular, $\rho_s(0)$ is the fraction of problems on which method $s$ achieves the lowest cost, and $\rho_s(1)$ and $\rho_s(2)$ give the fractions of problems solved within two and four times the lowest cost, respectively. To avoid zero denominators, successful runtimes are
bounded below by $10^{-6}$ seconds and successful oracle counts by one. \Cref{fig:numerical-mgh} shows the performance profiles over $0\le\alpha\le\log_2(100)$. For runtime, \texttt{HSODM} and \texttt{UTR} perform well near $\alpha=0$, while \texttt{APPNCG} achieves a high fraction of successful problems as the allowed cost
factor increases. For oracle evaluations, \texttt{AR3-Interp+} achieves
the lowest cost on many problems.

\begin{figure}[htbp]
    \centering
    \includegraphics[width=0.49\linewidth]{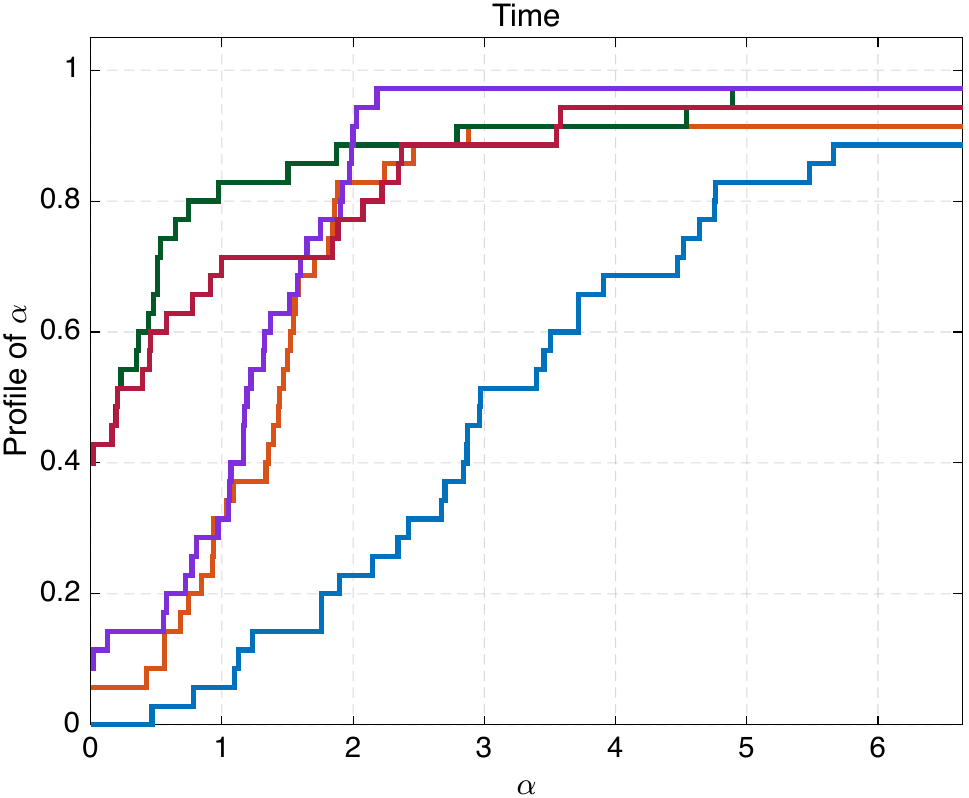}\hfill
    \includegraphics[width=0.49\linewidth]{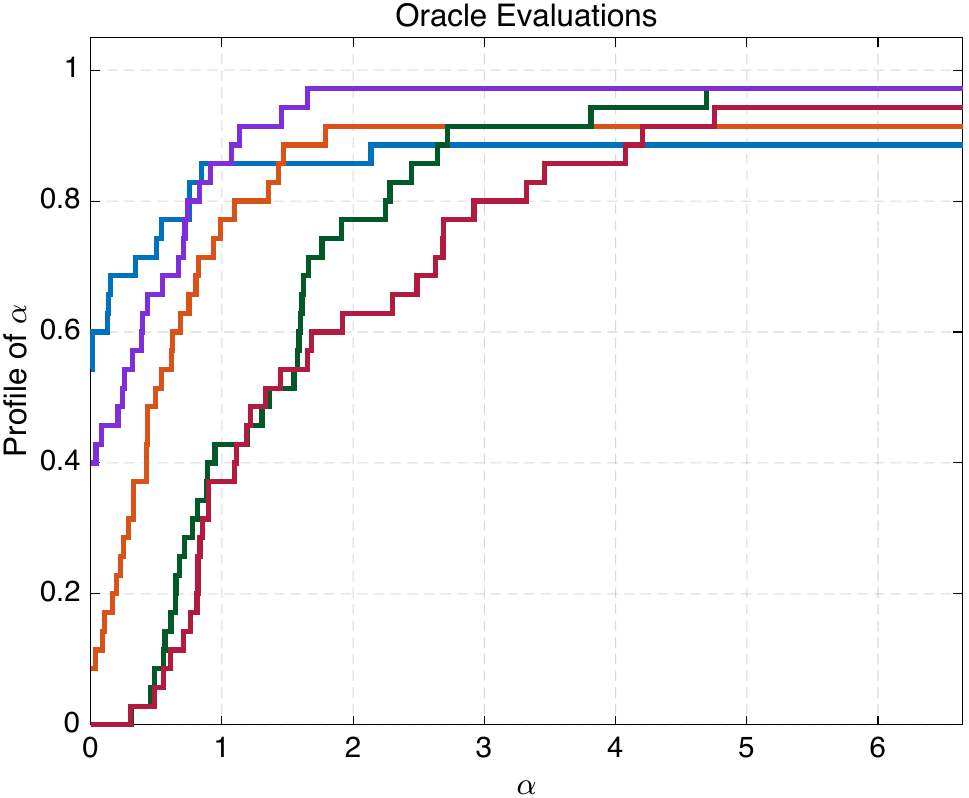}
    \par\vspace{5pt}
    \includegraphics[width=0.75\linewidth]{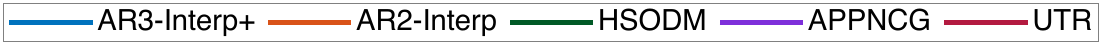}
    \caption{Performance profiles on the 35 MGH problems. The left panel uses running time and the right panel uses Oracle Evaluations defined in \eqref{eq:numerical-oracle-evaluations}. Both panels use the logarithmic performance ratio and the profile defined in \eqref{eq:numerical-performance-profile}.}
    \label{fig:numerical-mgh}
\end{figure}

We next consider the regularized sigmoid regression problem
\begin{equation}\label{eq:numerical-regression}
    \textstyle \min_{x\in\bR^d} \ f(x)\coloneqq
    \frac{1}{n}\sum_{i=1}^{n}\frac{1}{1+\exp(b_i\inner{a_i}{x})}
    +\frac{\lambda}{2}\|x\|^2,
\end{equation}
where $a_i\in\bR^d$ is a feature vector and $b_i\in\{-1,1\}$ is its label. The data-fitting term is the sigmoid loss and is generally nonconvex. We set $\lambda=10^{-4}$ and use six binary classification datasets obtained from the LIBSVM repository.\footnote{The datasets are available at \url{https://www.csie.ntu.edu.tw/~cjlin/libsvmtools/datasets/binary.html}.} For every dataset and every method, the initial point is $x^0=10\mathbf{1}_d$, where $\mathbf{1}_d$ denotes the all-ones vector in $\bR^d$. We use the same stationarity tolerances as in the MGH experiment and impose a total wall-clock limit of 300 seconds per run. \Cref{fig:numerical-regression-first,fig:numerical-regression-second} show the gradient norm against CPU time and Oracle Evaluations. \texttt{APPNCG} requires the least CPU time and the fewest Oracle Evaluations on every dataset. The difference is particularly pronounced on \texttt{covtype.binary}: \texttt{APPNCG} uses 3.72 CPU seconds and 339 Oracle Evaluations, compared with 13.51 seconds and 1,210 evaluations for \texttt{AR2-Interp} and 193.59 seconds and 1,206 evaluations for \texttt{AR3-Interp+}.

\begin{figure}[htbp]
    \centering
    \includegraphics[width=0.32\linewidth]{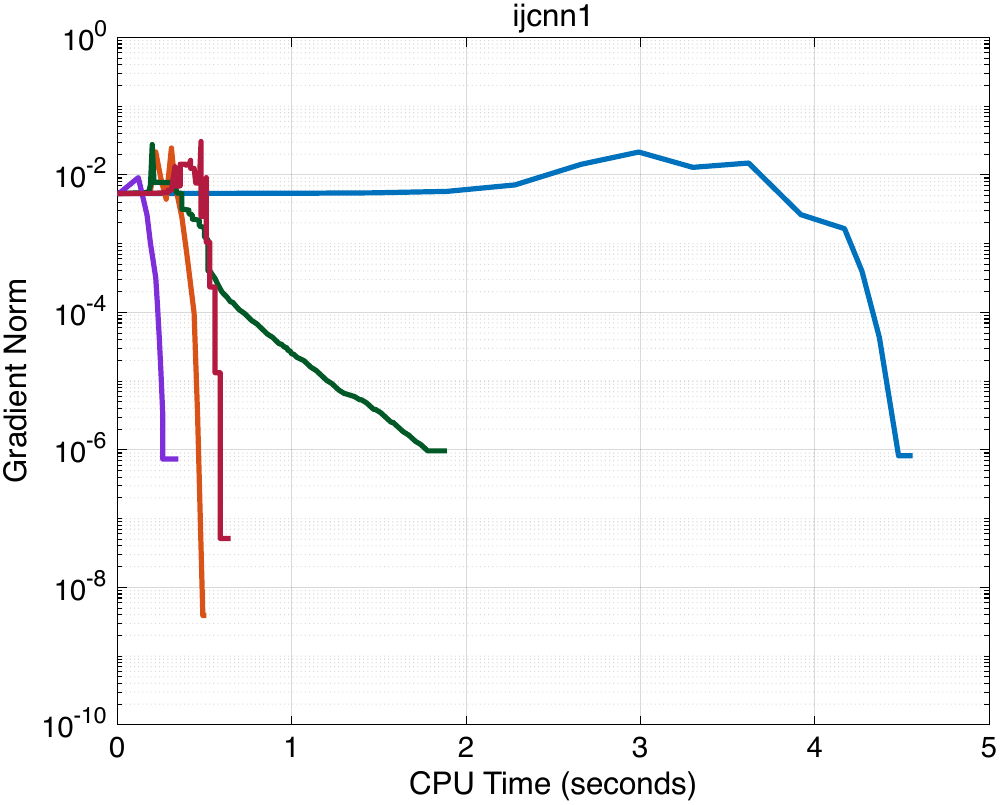}\hfill
    \includegraphics[width=0.32\linewidth]{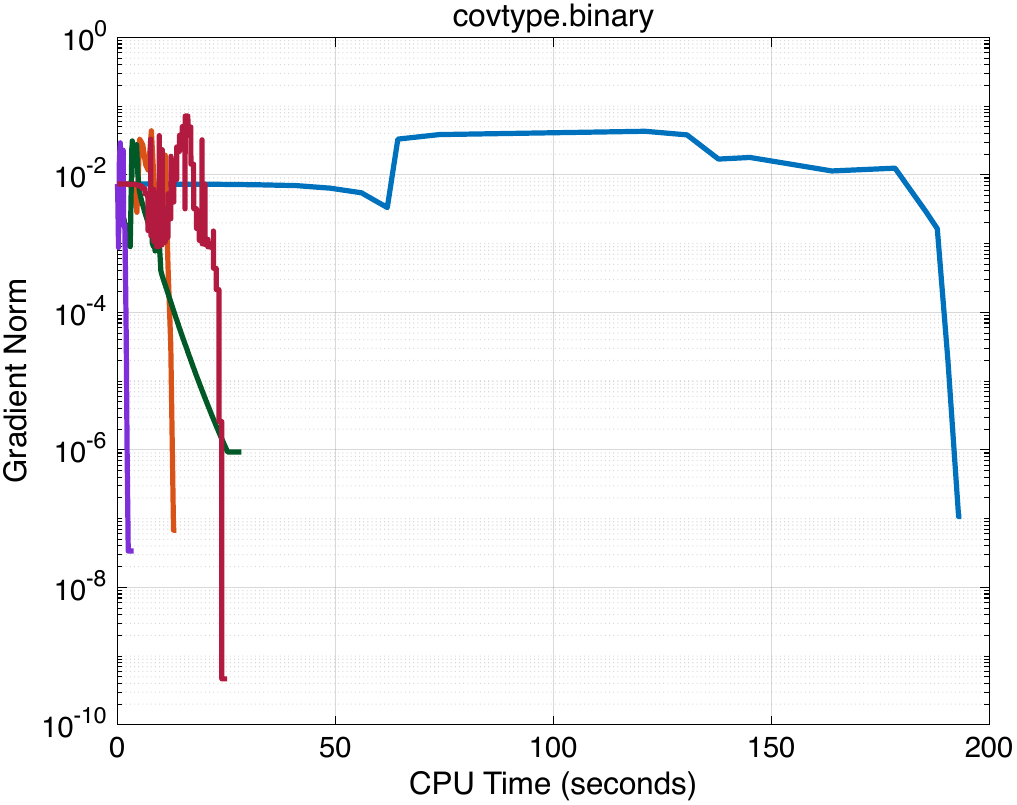}\hfill
    \includegraphics[width=0.32\linewidth]{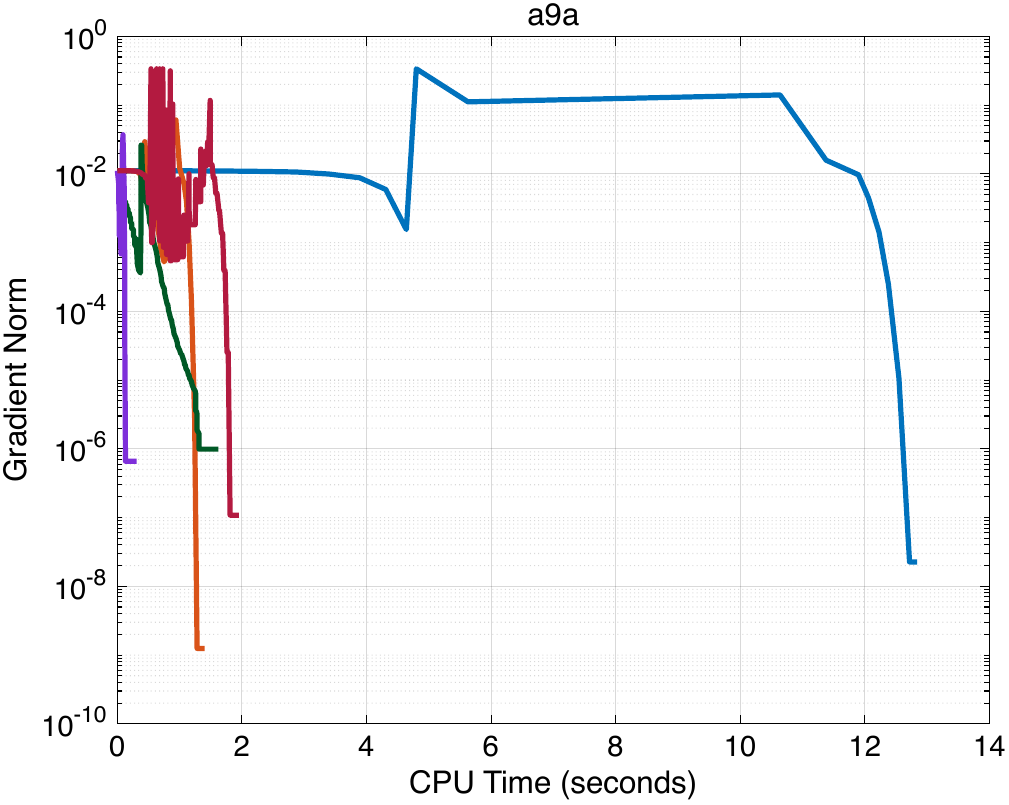}
    \par\vspace{4pt}
    \includegraphics[width=0.32\linewidth]{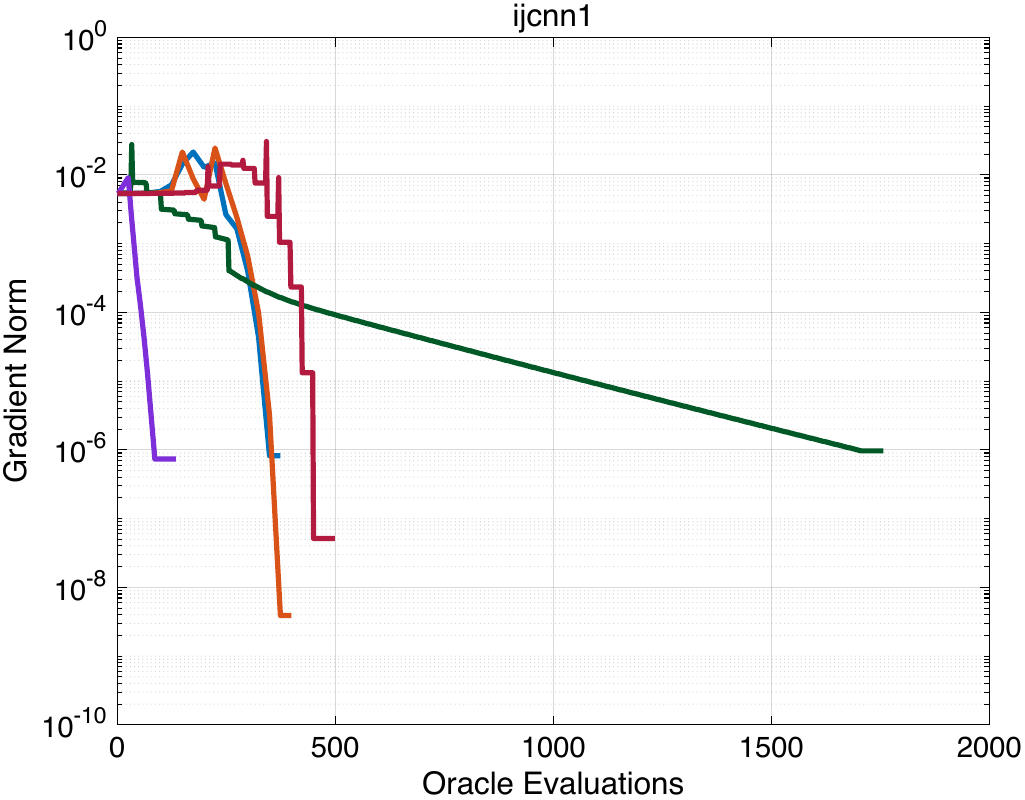}\hfill
    \includegraphics[width=0.32\linewidth]{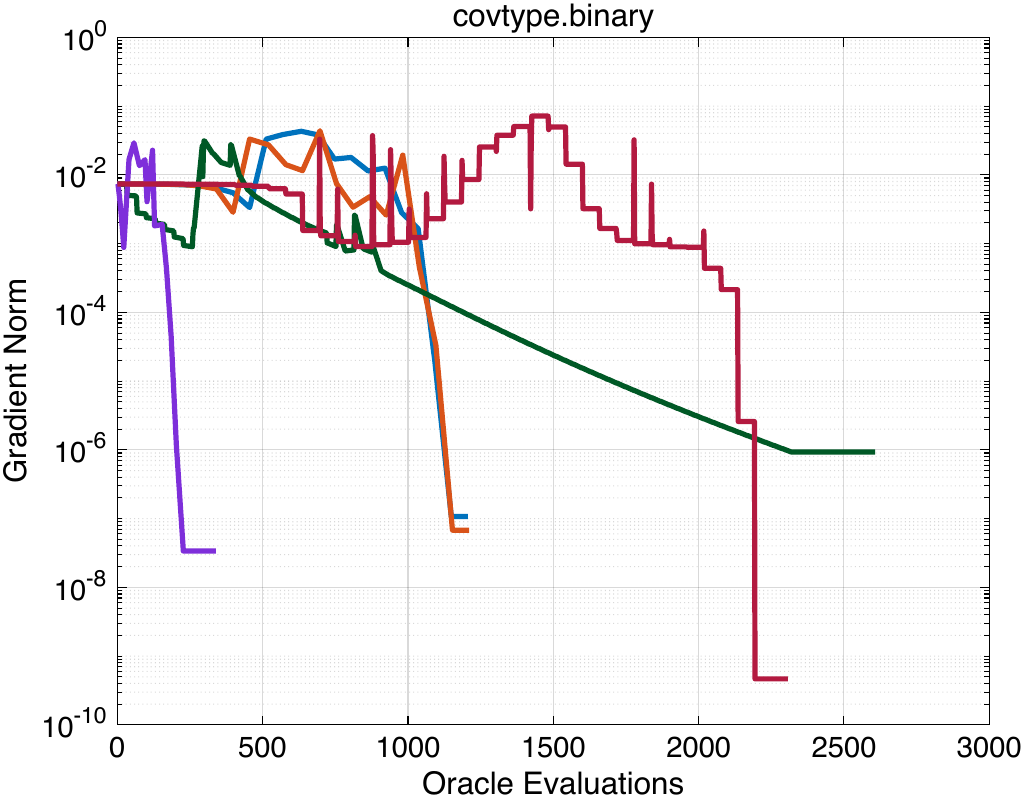}\hfill
    \includegraphics[width=0.32\linewidth]{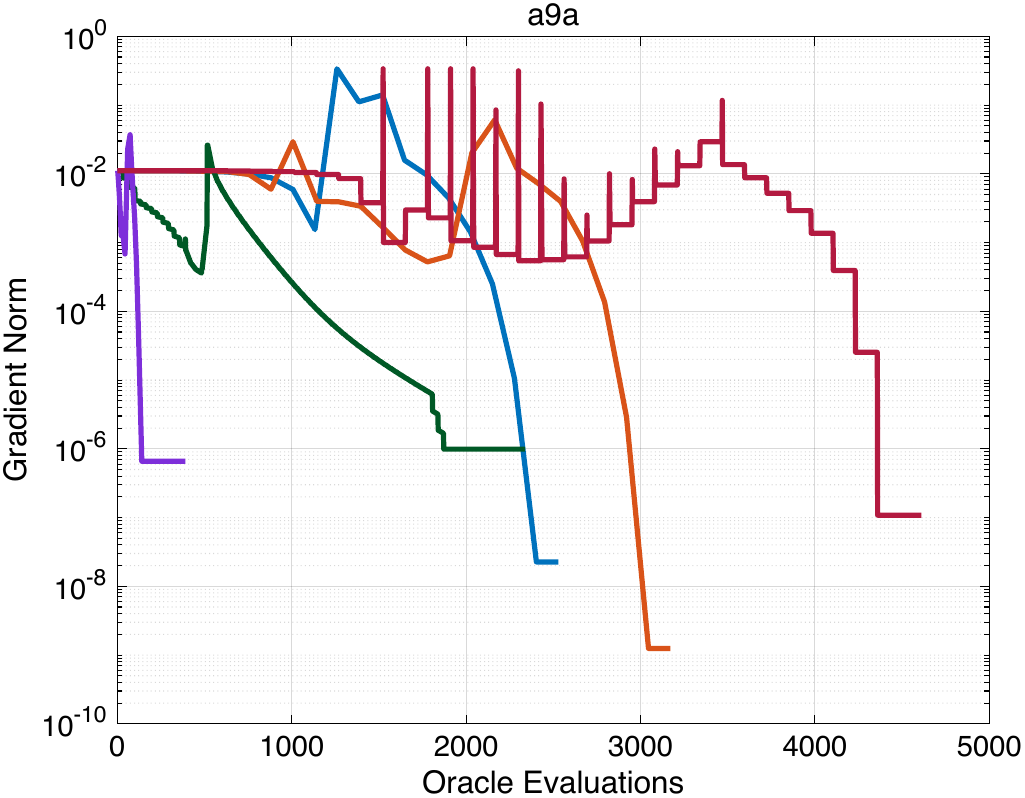}
    \par\vspace{5pt}
    \includegraphics[width=0.75\linewidth]{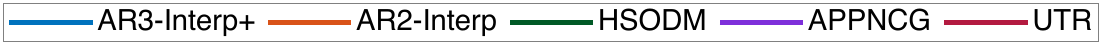}
    \caption{Gradient norms for \eqref{eq:numerical-regression} on \texttt{ijcnn1}, \texttt{covtype.binary}, and \texttt{a9a}. The top row uses CPU time and the bottom row uses Oracle Evaluations.}
    \label{fig:numerical-regression-first}
\end{figure}

\begin{figure}[htbp]
    \centering
    \includegraphics[width=0.32\linewidth]{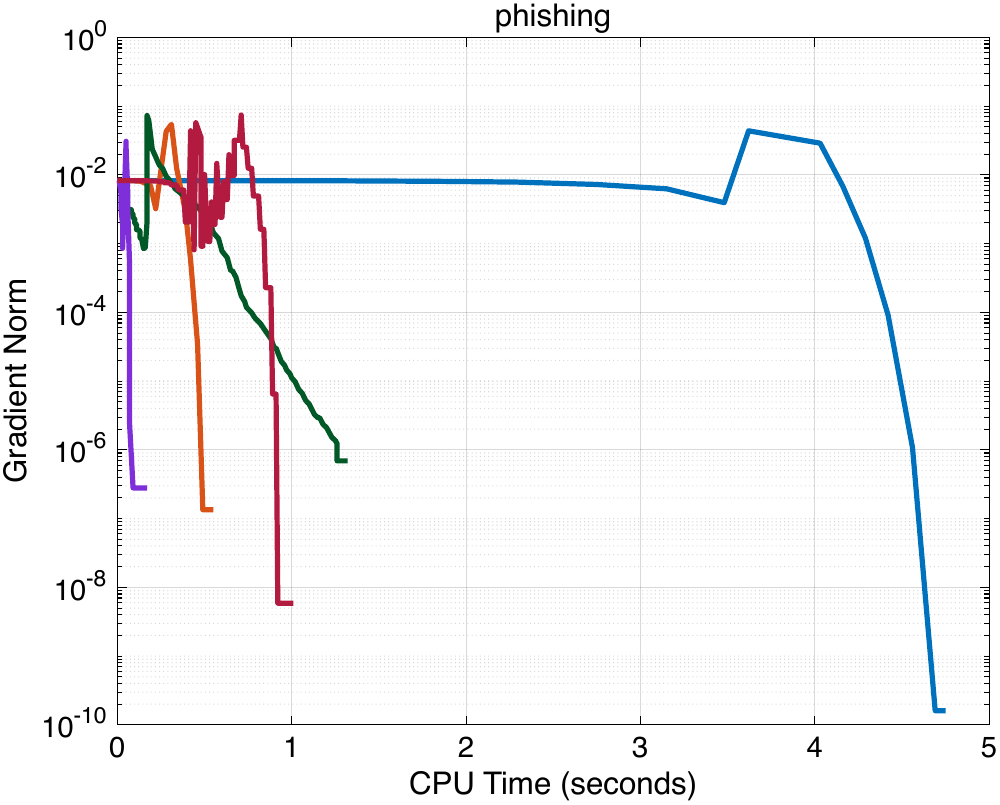}\hfill
    \includegraphics[width=0.32\linewidth]{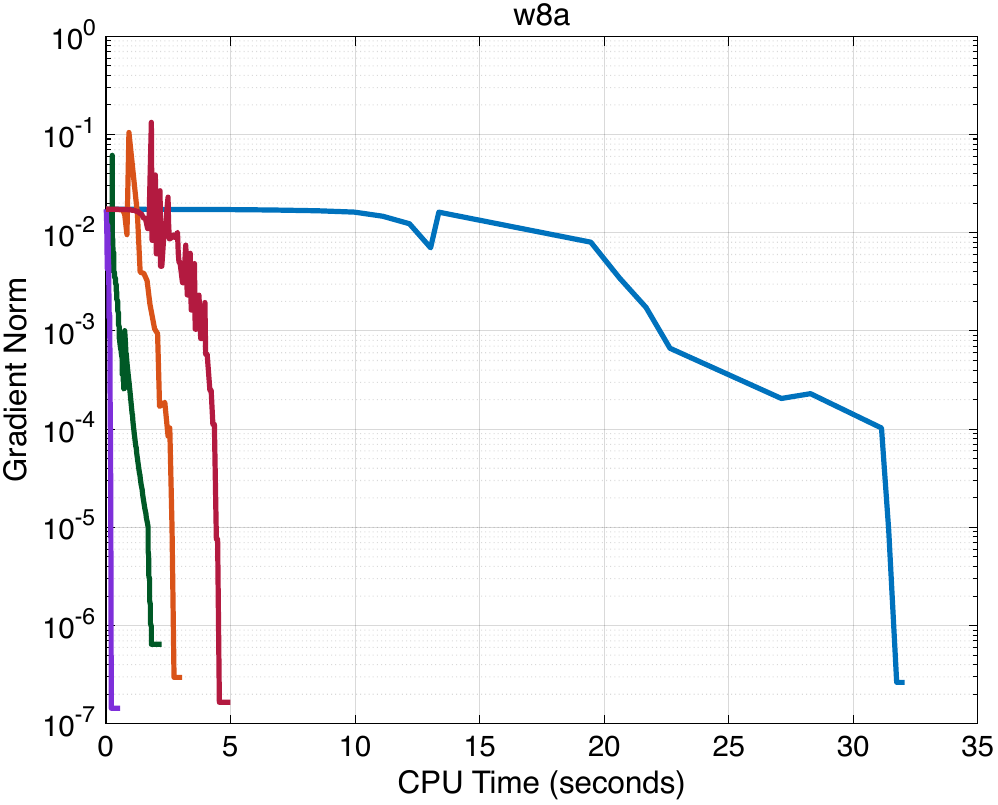}\hfill
    \includegraphics[width=0.32\linewidth]{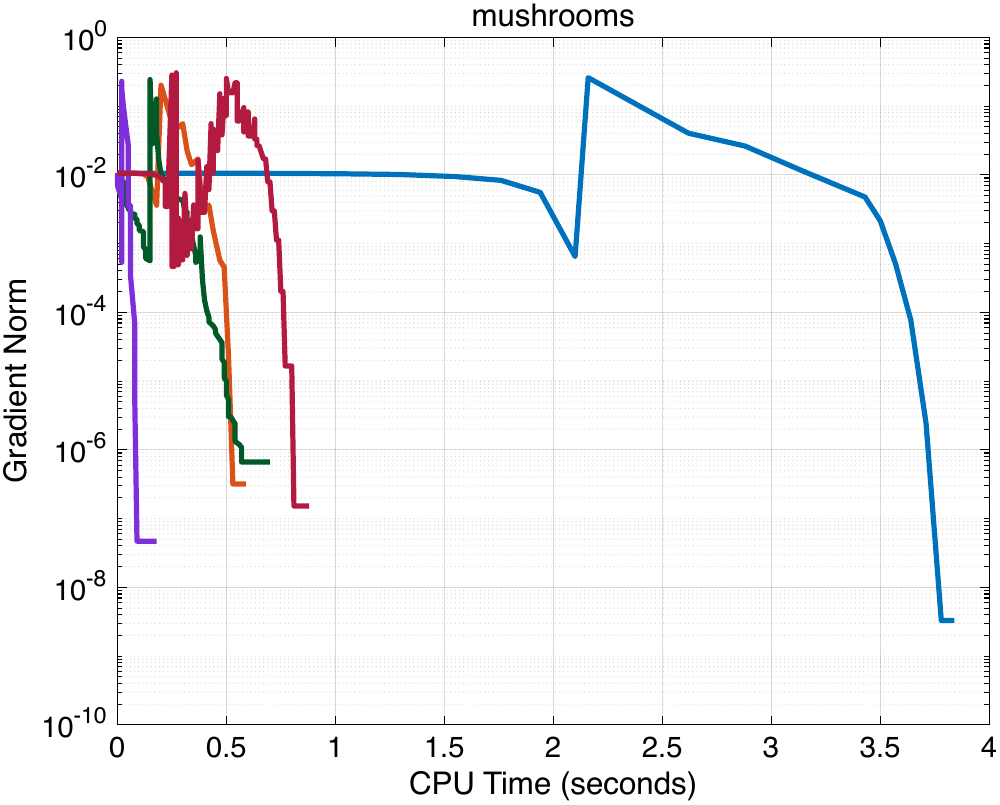}
    \par\vspace{4pt}
    \includegraphics[width=0.32\linewidth]{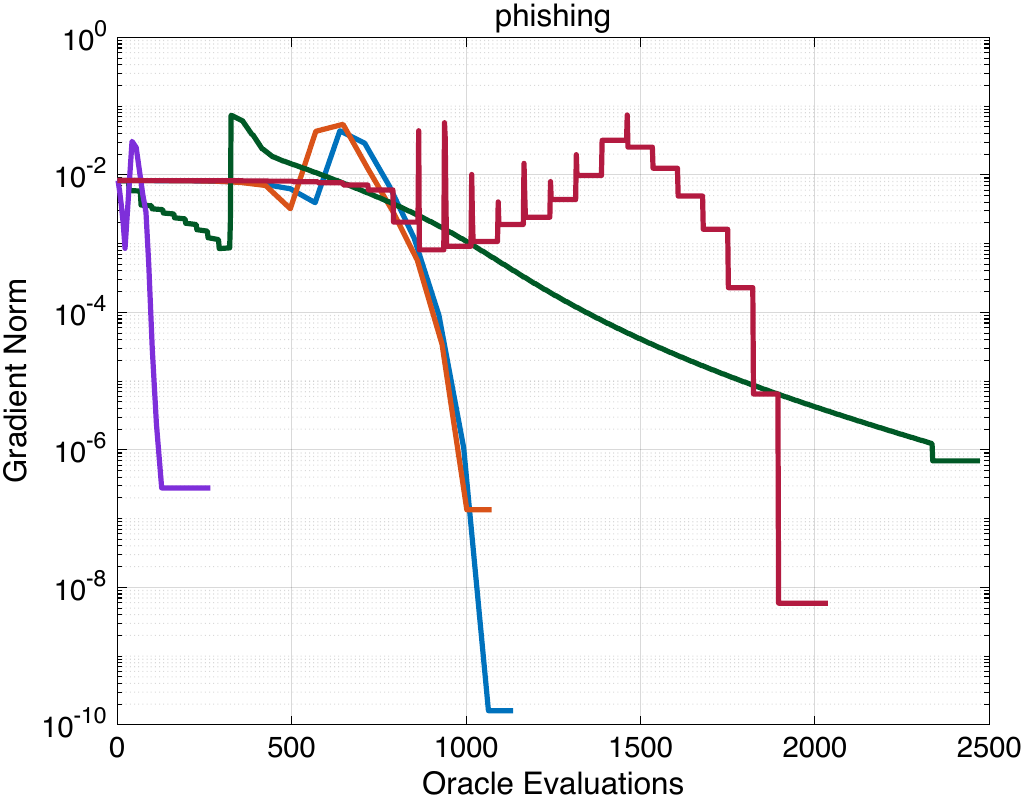}\hfill
    \includegraphics[width=0.32\linewidth]{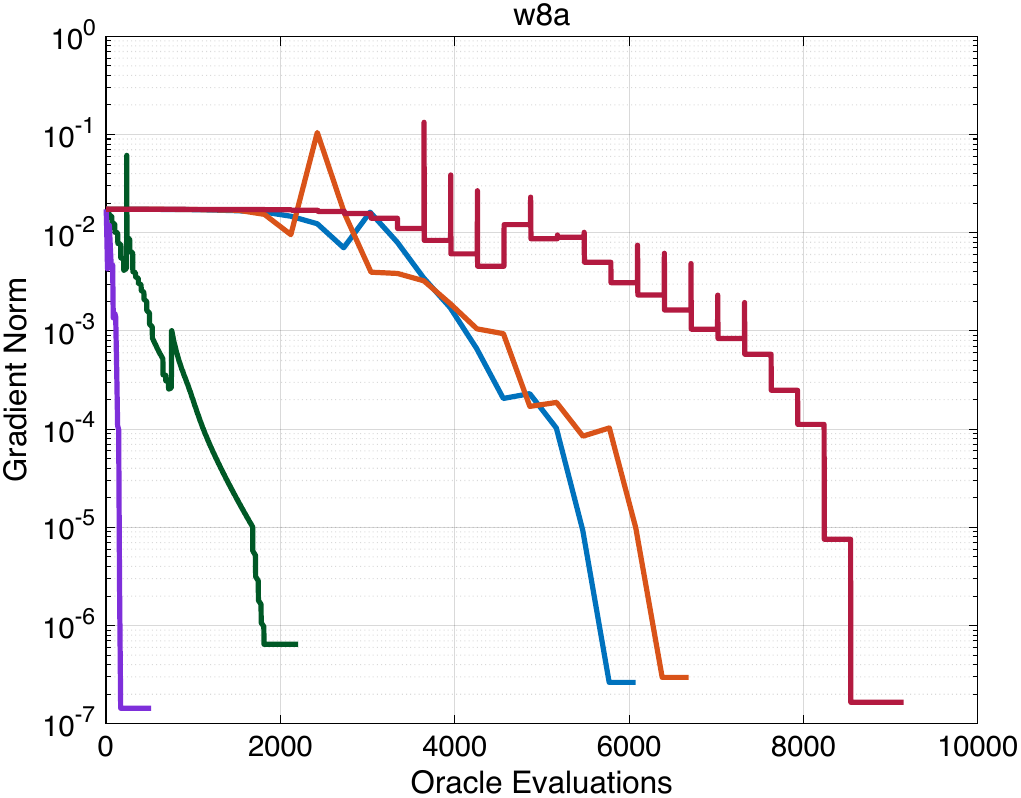}\hfill
    \includegraphics[width=0.32\linewidth]{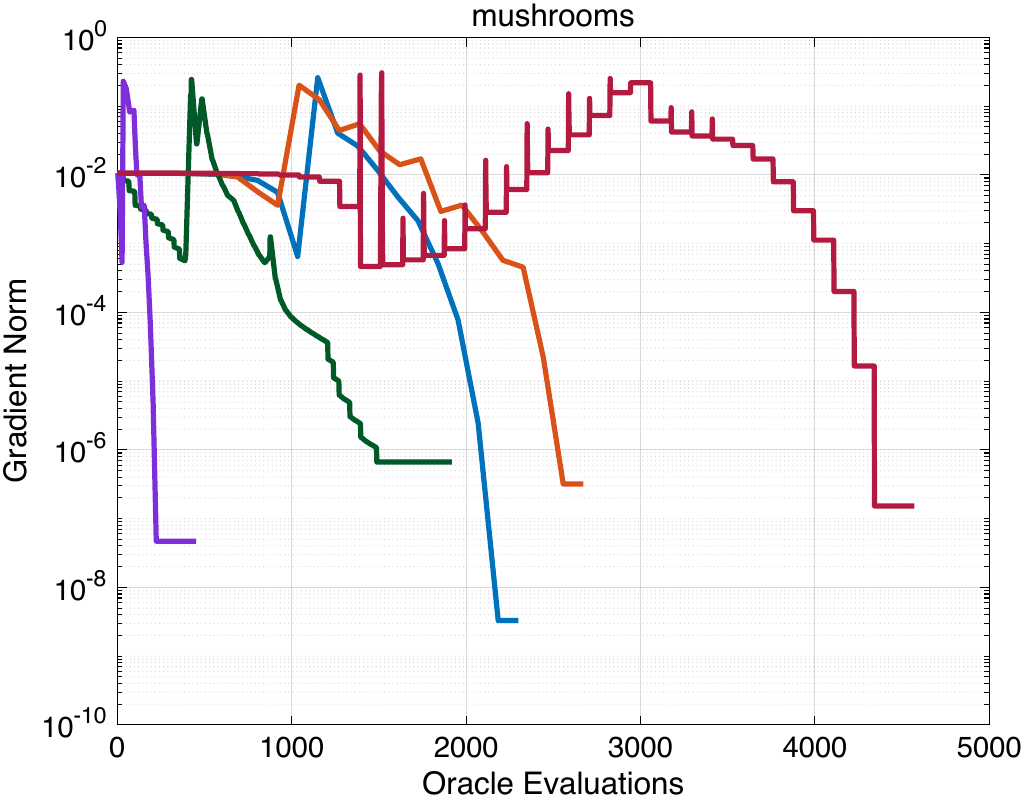}
    \par\vspace{5pt}
    \includegraphics[width=0.75\linewidth]{Figures/regression_legend.pdf}
    \caption{Gradient norms for \eqref{eq:numerical-regression} on \texttt{phishing}, \texttt{w8a}, and \texttt{mushrooms}. The top row uses CPU time and the bottom row uses Oracle Evaluations.}
    \label{fig:numerical-regression-second}
\end{figure}

\section{Concluding Remarks}

In this paper, we have developed a family of proximal point Newton methods using only first- and second-order information. We first established a single-loop framework matching the evaluation complexity of cubic regularized Newton methods, then introduced acceleration to exploit third-derivative smoothness and match the complexity of third-order tensor methods up to a logarithmic factor. Finally, we developed an adaptive Newton--CG implementation that relies on Hessian--vector products and requires no prior knowledge of Lipschitz constants. This stepwise development allows us to understand the respective roles of the proximal point framework, acceleration, and computational adaptivity. Several interesting questions remain open. First, it is worth investigating whether the logarithmic term can be removed and whether a single-loop (or even branch-free) version of \Cref{alg:accppncg} can be developed. Besides, it is also interesting to establish a lower bound on second-order oracle evaluations for the scenario in which the objective function is lower bounded, both the Hessian and third derivatives are Lipschitz continuous, but only a second-order oracle is used.

\bibliographystyle{plainnat}
\bibliography{ref}

\appendix
\crefalias{section}{appendix}

\section{Tensor-free inequalities}\label{app:hermite}
\begin{proof}[Proof of \Cref{lem:tensor-free}]
First, integration by parts gives
\begin{align*}
    &\textstyle f(y)-f(x)-\frac12\inner{\nabla f(x)+\nabla f(y)}{y-x}+\frac1{12}\inner{(\nabla^2f(y)-\nabla^2f(x))(y-x)}{y-x}\\
    =\ &\textstyle\frac1{12}\inner{(\nabla^2f(y)-\nabla^2f(x))(y-x)}{y-x}-\int_0^1\big(\tau-\frac12\big)\inner{\nabla^2f(x+\tau(y-x))(y-x)}{y-x}\mathrm{d}\tau\\
    =\ &\textstyle\int_0^1\frac{1-6\tau+6\tau^2}{12}\nabla^3f(x+\tau(y-x))[y-x]^{\otimes 3}\mathrm{d}\tau.
\end{align*}
Consequently, it suffices to bound the integral on the right-hand side. Note that the coefficient $\frac{1-6\tau+6\tau^2}{12}$ is symmetric about $\tau=1/2$. We thus have
\begin{align*}
    &\textstyle\big|\int_0^1\frac{1-6\tau+6\tau^2}{12}\nabla^3f(x+\tau(y-x))[y-x]^{\otimes 3}\mathrm{d}\tau\big|\\
    =\ &\textstyle\big|\int_0^{1/2}\frac{1-6\tau+6\tau^2}{12}\big(\nabla^3f(x+\tau(y-x))+\nabla^3f(y-\tau(y-x))\big)[y-x]^{\otimes 3}\mathrm{d}\tau\big|\\
    =\ &\textstyle\frac12\big|\int_0^{1/2}\int_0^{1/2}(\tau-s)(\tau+s-1)
    \big(\nabla^3f(x+\tau(y-x))-\nabla^3f(x+s(y-x))\\
    &\textstyle+\nabla^3f(y-\tau(y-x))-\nabla^3f(y-s(y-x))\big)
    [y-x]^{\otimes 3}\mathrm{d}s\mathrm{d}\tau\big|\\
    \le\ &\textstyle\frac{\|y-x\|^3}{2}\int_0^{1/2}\int_0^{1/2}|\tau-s|(1-\tau-s)
    \big(\|\nabla^3f(x+\tau(y-x))-\nabla^3f(x+s(y-x))\|_{\op}\\
    &+\|\nabla^3f(y-\tau(y-x))-\nabla^3f(y-s(y-x))\|_{\op}\big)
    \mathrm{d}s\mathrm{d}\tau.
\end{align*}
By \Cref{ass:third}, for $0\le s,\tau\le1/2$, it follows that
\begin{align*}
    &\textstyle\|\nabla^3f(x+\tau(y-x))-\nabla^3f(x+s(y-x))\|_{\op}\le L_T|\tau-s|\|y-x\|,\\
    &\textstyle\|\nabla^3f(y-\tau(y-x))-\nabla^3f(y-s(y-x))\|_{\op}\le L_T|\tau-s|\|y-x\|,
\end{align*}
which implies
\begin{align*}
    &\textstyle\big|f(y)-f(x)-\frac12\inner{\nabla f(x)+\nabla f(y)}{y-x}+\frac1{12}\inner{(\nabla^2f(y)-\nabla^2f(x))(y-x)}{y-x}\big|\\
    \le\ &\textstyle L_T\|y-x\|^4\int_0^{1/2}\int_0^{1/2}(\tau-s)^2(1-\tau-s)\mathrm{d}s\mathrm{d}\tau\\
    =\ &\textstyle\frac{L_T}{192}\|y-x\|^4.
\end{align*}
The inequality \eqref{eq:tensor-free-function-taylor} is proved. It remains to show \eqref{eq:chord}. For $0\le\tau\le1$, we have
\begin{align*}
    &\textstyle\|\nabla^2f((1-\tau)x+\tau y)-(1-\tau)\nabla^2f(x)-\tau\nabla^2f(y)\|_{\op}\\
    \le\ &\textstyle(1-\tau)\|\nabla^2f(x)-\nabla^2f((1-\tau)x+\tau y)+\tau\nabla^3f((1-\tau)x+\tau y)[y-x,\cdot,\cdot]\|_{\op}\\
    &+\tau\|\nabla^2f(y)-\nabla^2f((1-\tau)x+\tau y)-(1-\tau)\nabla^3f((1-\tau)x+\tau y)[y-x,\cdot,\cdot]\|_{\op},
\end{align*}
where the third-order terms cancel since $(1-\tau)\tau-\tau(1-\tau)=0$. By applying \eqref{eq:taylor} and using $x-((1-\tau)x+\tau y)=-\tau(y-x)$ and $y-((1-\tau)x+\tau y)=(1-\tau)(y-x)$, it follows that
\begin{align*}
    &\textstyle\|\nabla^2f(x)-\nabla^2f((1-\tau)x+\tau y)+\tau\nabla^3f((1-\tau)x+\tau y)[y-x,\cdot,\cdot]\|_{\op}\le\frac{L_T}{2}\tau^2\|y-x\|^2,\\
    &\textstyle\|\nabla^2f(y)-\nabla^2f((1-\tau)x+\tau y)-(1-\tau)\nabla^3f((1-\tau)x+\tau y)[y-x,\cdot,\cdot]\|_{\op}\le\frac{L_T}{2}(1-\tau)^2\|y-x\|^2,
\end{align*}
which implies 
\begin{align*}
    \textstyle\|\nabla^2f((1-\tau)x+\tau y)-(1-\tau)\nabla^2f(x)-\tau\nabla^2f(y)\|_{\op} \le \ &\textstyle\frac{L_T}{2}\big((1-\tau)\tau^2+\tau(1-\tau)^2\big)\|y-x\|^2 \\
    = \ &\textstyle \frac{L_T}{2}\tau(1-\tau)\|y-x\|^2.
\end{align*}
The proof is complete.
\end{proof}

\section{Capped Conjugate Gradient}\label[appendix]{app:capped-cg}
We give the capped CG procedure used in \eqref{eq:ncg-cg-call}, following Algorithm~1 of \cite{royeroneillwright2018}. For symmetric $H\in\bR^{d\times d}$, $g\in\bR^d$, and $\mu>0$, the procedure approximately solves $(H+2\mu I)y=-g$ or returns a nonzero direction of negative curvature for $H$. We denote the CG iterates, residuals, and search directions by $y^j$, $r^j=(H+2\mu I)y^j+g$, and $p^j$, respectively. The output types are \texttt{SOL} and \texttt{NC}, as in \Cref{alg:accppncg}. The capped CG procedure does not require an a priori upper bound on $\|H\|_{\op}$. Instead, it maintains a scalar $M$, initialized to zero, to bound the ratios $\|Hv\|/\|v\|$ observed during CG for nonzero vectors $v$. Whenever a tested vector $v$ satisfies $\|Hv\|>M\|v\|$, we set $M=\|Hv\|/\|v\|$ and recompute
\begin{equation}\label{eq:ccg-parameters}
    \textstyle\kappa=\frac{M+2\mu}{\mu},\ \hat\zeta=\frac{\zeta}{3\kappa},\ \tau=\frac{\sqrt\kappa}{\sqrt\kappa+1},\ T=\frac{4\kappa^4}{(1-\sqrt\tau)^2}.
\end{equation}
If the preceding curvature test does not terminate CG, the residual test $\|r^j\|\le\hat\zeta\|g\|$ returns $(\texttt{SOL},y^j)$. If instead $\|r^j\|>\sqrt T\tau^{j/2}\|g\|$ and none of the preceding tests terminates CG, one further CG update produces $y^{j+1}$, and Theorem~2 of \cite{royeroneillwright2018} guarantees an index $i\in\{0,\ldots,j-1\}$ such that
\begin{equation}\label{eq:ccg-negative-curvature}
    \textstyle\inner{(H+2\mu I)(y^{j+1}-y^i)}{y^{j+1}-y^i}<\mu\|y^{j+1}-y^i\|^2.
\end{equation}
The last branch of \Cref{procedure:capped-cg} then returns $(\texttt{NC},y^{j+1}-y^i)$. The following lemma collects the capped-CG guarantees of \cite{royeroneillwright2018}. For the call in \eqref{eq:ncg-cg-call}, we have $H+2\mu I=\nabla^2\Phi_k(x^{k,t})$, $g=-\nabla\Phi_k(x^{k,t})$, and $\zeta=1/32$. Therefore, the \texttt{SOL} output $d^{k,t}=y$ satisfies \eqref{eq:binary-sol}. For a \texttt{NC} output, setting $u^{k,t}=y/\|y\|$ gives \eqref{eq:binary-nc}.

\begin{lemma}\label{lem:ccg-interface}
In exact arithmetic, \Cref{procedure:capped-cg} terminates deterministically. Its \texttt{SOL} output satisfies
\begin{equation*}
    \inner{(H+2\mu I)y}{y}\ge\mu\|y\|^2,\ \|(H+2\mu I)y+g\|\le\min\Big\{\frac{\zeta\mu}{2}\|y\|,\frac{\zeta}{6}\|g\|\Big\}.
\end{equation*}
Its \texttt{NC} output satisfies $y\ne0$ and $\inner{Hy}{y}<-\mu\|y\|^2$. The number of matrix-vector products is
\begin{equation*}
    \widetilde\cO\big(\min\{d,\sqrt{2+\|H\|_{\op}/\mu}\}\big).
\end{equation*}
\end{lemma}

\begin{procedure}[H]
\caption{Capped Conjugate Gradient}\label{procedure:capped-cg}
\begin{algorithmic}[1]
\Require Symmetric $H\in\bR^{d\times d}$, $g\in\bR^d$, regularization parameter $\mu>0$, relative accuracy $\zeta\in(0,1)$.
\Ensure $(\texttt{SOL},y)$ or $(\texttt{NC},y)$ satisfying \Cref{lem:ccg-interface}.
\State \textbf{Initialization:} $M=0$, $y^0=0$, $r^0=g$, $p^0=-g$, $j=0$; compute \eqref{eq:ccg-parameters}.
\If{$g=0$}
  \State \Return $(\texttt{SOL},0)$.
\EndIf
\If{$\inner{(H+2\mu I)p^0}{p^0}<\mu\|p^0\|^2$}
  \State \Return $(\texttt{NC},p^0)$.
\EndIf
\State Set $M=\|Hp^0\|/\|p^0\|$ and recompute \eqref{eq:ccg-parameters}.
\While{true}
  \State $\alpha_j=\|r^j\|^2/\inner{(H+2\mu I)p^j}{p^j}$.
  \State $y^{j+1}=y^j+\alpha_jp^j$, $r^{j+1}=r^j+\alpha_j(H+2\mu I)p^j$.
  \State $\beta_{j+1}=\|r^{j+1}\|^2/\|r^j\|^2$, $p^{j+1}=-r^{j+1}+\beta_{j+1}p^j$, $j=j+1$.
  \For{$v\in\{p^j,y^j,r^j\}$}
    \If{$\|Hv\|>M\|v\|$}
      \State Set $M=\|Hv\|/\|v\|$ and recompute \eqref{eq:ccg-parameters}.
    \EndIf
  \EndFor
  \If{$\inner{(H+2\mu I)y^j}{y^j}<\mu\|y^j\|^2$}
    \State \Return $(\texttt{NC},y^j)$.
  \ElsIf{$\|r^j\|\le\hat\zeta\|g\|$}
    \State \Return $(\texttt{SOL},y^j)$.
  \ElsIf{$\inner{(H+2\mu I)p^j}{p^j}<\mu\|p^j\|^2$}
    \State \Return $(\texttt{NC},p^j)$.
  \ElsIf{$\|r^j\|>\sqrt T\tau^{j/2}\|g\|$}
    \State $\alpha_j=\|r^j\|^2/\inner{(H+2\mu I)p^j}{p^j}$, $y^{j+1}=y^j+\alpha_jp^j$.
    \State Find $i\in\{0,\ldots,j-1\}$ satisfying \eqref{eq:ccg-negative-curvature}.
    \State \Return $(\texttt{NC},y^{j+1}-y^i)$.
  \EndIf
\EndWhile
\end{algorithmic}
\end{procedure}

\section{Minimum eigenvalue oracle and its implementation}\label[appendix]{app:meo}
We give the minimum eigenvalue oracle used in \eqref{eq:ncg-meo-call}, following Appendix~B.3 of \cite{royeroneillwright2018}. Its inputs are a symmetric matrix $H\in\bR^{d\times d}$, an accuracy tolerance $\varepsilon>0$, and a failure probability parameter $\rho\in(0,1)$. The oracle computes a unit vector $u$ and its Rayleigh quotient $\lambda=\inner{Hu}{u}$. It returns $(\texttt{NC},\lambda,u)$ when $\lambda\le-\varepsilon/2$ and otherwise returns $(\texttt{CERT},\lambda,u)$, asserting $\lambda_{\min}(H)\ge-\varepsilon$. Similarly, the oracle does not require an a priori upper bound on $\|H\|_{\op}$. It uses a single Lanczos run with a starting vector drawn independently and uniformly from the unit sphere. After $j$ updates, let $Q_j=[q^1,\ldots,q^j]$ and $T_j=Q_j^THQ_j$, where $Q_j$ has orthonormal columns and $T_j$ is symmetric tridiagonal. The first phase estimates $\|H\|_{\op}$ from the extreme eigenvalues of $T_j$, and the second phase continues the same run to obtain the required accuracy. The following lemma collects the corresponding guarantees of \cite{royeroneillwright2018}. 

\begin{procedure}[H]
\caption{Minimum Eigenvalue Oracle}\label{procedure:meo}
\begin{algorithmic}[1]
\Require Symmetric $H\in\bR^{d\times d}$, accuracy tolerance $\varepsilon>0$, failure probability parameter $\rho\in(0,1)$.
\Ensure $(\texttt{NC},\lambda,u)$ or $(\texttt{CERT},\lambda,u)$ satisfying \Cref{lem:meo-guarantee}.
\State \textbf{Initialization:} $q^0=0$, $\beta_0=0$, $j=0$; draw $q^1$ independently and uniformly from the unit sphere.
\State $j_M=\min\{d,1+\lceil\tfrac12\log(25d/\rho^2)\rceil\}$.
\While{$j<j_M$}
  \State Set $j=j+1$ and perform \Cref{procedure:lanczos-update}.
  \State \textbf{If} $\beta_j=0$, go to \textbf{Classification}.
\EndWhile
\State $M=2\|T_{j_M}\|_{\op}=2\max\{|\lambda_{\min}(T_{j_M})|,|\lambda_{\max}(T_{j_M})|\}$.
\State $m=\min\{d,\max\{j_M,1+\lceil\tfrac12\log(25d/\rho^2)\sqrt{M/\varepsilon}\rceil\}\}$.
\While{$j<m$}
  \State Set $j=j+1$ and perform \Cref{procedure:lanczos-update}.
  \State \textbf{If} $\beta_j=0$, go to \textbf{Classification}.
\EndWhile
\State \textbf{Classification:} compute $\lambda=\lambda_{\min}(T_j)$ and $z$ with $T_jz=\lambda z$, $\|z\|=1$; set $u=Q_jz$.
\If{$\lambda\le-\varepsilon/2$}
  \State \Return $(\texttt{NC},\lambda,u)$.
\Else
  \State \Return $(\texttt{CERT},\lambda,u)$, certifying $\lambda_{\min}(H)\ge-\varepsilon$.
\EndIf
\end{algorithmic}
\end{procedure}

\begin{procedure}[H]
\caption{Lanczos Update}\label{procedure:lanczos-update}
\begin{algorithmic}[1]
\Require Symmetric $H\in\bR^{d\times d}$, index $1\le j\le d$, orthonormal vectors $q^1,\ldots,q^j$, $q^0=0$, $\beta_0=0$.
\Ensure Updated $Q_j$, $T_j$, $\beta_j$, and $q^{j+1}$ when $\beta_j>0$ and $j<d$.
\State $w=Hq^j-\beta_{j-1}q^{j-1}$.
\State $\alpha_j=\inner{w}{q^j}$, $w=w-\alpha_jq^j$.
\For{$i=1,\ldots,j$}
  \State $w=w-\inner{w}{q^i}q^i$.
\EndFor
\State $\beta_j=\|w\|$, $Q_j=[q^1,\ldots,q^j]$.
\State Form symmetric tridiagonal $T_j$ with diagonal $\alpha_1,\ldots,\alpha_j$ and off-diagonal $\beta_1,\ldots,\beta_{j-1}$.
\If{$\beta_j>0$ and $j<d$}
  \State $q^{j+1}=w/\beta_j$.
\EndIf
\end{algorithmic}
\end{procedure}

\begin{lemma}\label{lem:meo-guarantee}
In exact arithmetic, for fixed inputs $H=H^\top$, $\varepsilon>0$, and $\rho\in(0,1)$, with probability at least $1-\rho$, \Cref{procedure:meo} terminates and returns a Ritz pair satisfying
\begin{equation}\label{eq:meo-accuracy}
    \|u\|=1,\ \lambda=\inner{Hu}{u},\ 0\le\lambda-\lambda_{\min}(H)\le\varepsilon/2.
\end{equation}
Its \texttt{NC} output always satisfies $\inner{Hu}{u}\le-\varepsilon/2$, and an incorrect \texttt{CERT} has probability at most $\rho$. The number of matrix-vector products with $H$ is
\begin{equation*}
    \widetilde\cO\big(\min\{d,\sqrt{1+\|H\|_{\op}/\varepsilon}\}\big).
\end{equation*}
\end{lemma}

For the call in \eqref{eq:ncg-meo-call}, we have $\varepsilon=\mu/16$. If \texttt{NC} is returned for a sign $s$, set $u^{k,t}=u$ and $s^{k,t}=s$. Since the matrix in \eqref{eq:ncg-gate} is $H+3\mu I/32$, the classification test gives
\begin{equation*}
    \textstyle\inner{(H+3\mu I/32)u}{u}=\lambda+3\mu/32\le\mu/16,
\end{equation*}
which yields \eqref{eq:ncg-gate}. If both signs return \texttt{CERT} and \eqref{eq:meo-accuracy} holds for both calls, then $\lambda_{\min}(H)\ge\lambda-\mu/32>-\mu/16$ for each matrix in \eqref{eq:ncg-meo-call}. Hence, $\lambda_{\min}(H+3\mu I/32)>\mu/32$ for both signs, which yields \eqref{eq:ncg-certificate}. For the optional call in \eqref{eq:ncg-linesearch-call}, we have $\varepsilon=\epsH/2$ and $H=\nabla^2f(x^k-\beta^j d^{k,0})+\epsH I/2$. If \texttt{CERT} is returned and \eqref{eq:meo-accuracy} holds for this call, then
\begin{equation*}
    \textstyle\lambda_{\min}\big(\nabla^2f(x^k-\beta^j d^{k,0})\big)=\lambda_{\min}(H)-\epsH/2\ge\lambda-3\epsH/4>-\epsH,
\end{equation*}
where the last inequality uses $\lambda>-\epsH/4$. All calls use $\rho=\delta/[\ell(\ell+1)]$. Then \eqref{eq:meo-accuracy} holds for all calls simultaneously with probability at least $1-\delta$, as shown in the proof of \Cref{thm:binary}.

\end{document}